\documentclass[12pt]{amsart}
\usepackage{txfonts}
\usepackage{amssymb}
\usepackage{eucal}
\usepackage{graphicx}
\usepackage{amsmath}
\usepackage{amscd}
\usepackage[all]{xy}           

\usepackage{amsfonts,latexsym}
\usepackage{xspace}
\usepackage{epsfig}
\usepackage{float}
\usepackage{color}
\usepackage{fancybox}
\usepackage{colordvi}
\usepackage{multicol}
\usepackage{colordvi}
\usepackage{stmaryrd}

\newtheorem{theorem}{Theorem}[section]
\newtheorem{prop}[theorem]{Proposition}
\newtheorem{defn}[theorem]{Definition}
\newtheorem{lemma}[theorem]{Lemma}
\newtheorem{coro}[theorem]{Corollary}
\newtheorem{prop-def}{Proposition-Definition}[section]

\newtheorem{remark}[theorem]{Remark}

\newtheorem{exam}[theorem]{Example}

\newcommand{\nc}{\newcommand}
\newcommand{\delete}[1]{}

\nc{\mlabel}[1]{\label{#1}}  
\nc{\mcite}[1]{\cite{#1}}  
\nc{\mref}[1]{\ref{#1}}  
\nc{\mbibitem}[1]{\bibitem{#1}} 

\delete{
\nc{\mlabel}[1]{\label{#1}  
{\hfill \hspace{1cm}{\bf{{\ }\hfill(#1)}}}}
\nc{\mcite}[1]{\cite{#1}{{\bf{{\ }(#1)}}}}  
\nc{\mref}[1]{\ref{#1}{{\bf{{\ }(#1)}}}}  
\nc{\mbibitem}[1]{\bibitem[\bf #1]{#1}} 
}

\nc{\bfk}{\mathbf{k}}
\nc{\Der}{\mathrm{Der}}
\nc{\Ker}{\mathrm{Ker}}

\begin{document}

\title{Hom 3-Lie-Rinehart Algebras }\footnotetext{ Corresponding author: Ruipu Bai, E-mail: bairuipu@hbu.edu.cn.}

\author{RuiPu  Bai}
\address{College of Mathematics and Information Science,
Hebei University
\\
Key Laboratory of Machine Learning and Computational\\ Intelligence of Hebei Province, Baoding 071002, P.R. China} \email{bairuipu@hbu.edu.cn}

\author{Xiaojuan Li}
\address{College of Mathematics and Information  Science,
Hebei University, Baoding 071002, China} \email{lixiaojuan1209@126.com}

\author{Yingli Wu}
\address{College of Mathematics and Information  Science,
Hebei University, Baoding 071002, China} \email{15733268503@163.com}

\date{}

\begin{abstract}

After endowing with a 3-Lie-Rinehart structure on Hom 3-Lie algebras, we obtain a class of special
 Hom 3-Lie algebras, which have close relationships with representations of commutative associative algebras. We provide a special class of Hom 3-Lie-Rinehart algebras, called  split regular Hom 3-Lie-Rinehart algebras, and we then characterize their structures by means of root systems and weight systems associated to a splitting Cartan subalgebra.

\end{abstract}

\subjclass[2010]{17B05, 17D99.}

\keywords{ Hom 3-Lie algebra,  3-Lie-Rinehart algebra,  split regular Hom 3-Lie-Rinehart algebra}

\maketitle



\allowdisplaybreaks

\section{Introduction}

The study of 3-Lie algebras  \cite{F} gets a lot of attention since it has close relationships with  Lie algebras, Hom-Lie algebras, commutative associative algebras, and cubic matrices \cite{B2,B3,B5, B4}. For example, it is applied to the study of Nambu mechanics and the study of supersymmetry and gauge symmetry transformations of the world-volume theory of multiple coincident M2-branes \cite{HHY,BLG2,Ro1,T}. The authors of \cite{BGL,BCR,B,ST,B1} studied the symplectic structures on 3-Lie algebras,  3-Lie bialgebras and 3-Lie Yang-Baxter equation, and constructed the tensor form of skew-symmetric solutions of 3-Lie Yang-Baxter equation in  3-Lie algebras. In \cite{26,27,28}, Casas introduced the cross modules of Lie-Rinehart algebras and proved that three cohomology and Rinehart cohomology are isomorphic if  Lie-Rinehart algebras are projected onto the corresponding commutative algebras. Ashis Mandal et al defined the Hom-Lie-Rinehart algebras and discussed its extension in the small dimension cohomology space \cite{29}.

In  this paper, we continue to study 3-Lie-Rinehart algebras introduced in \cite{B6}. We endow with a 3-Lie-Rinehart structure on Hom 3-Lie algebras, and obtain a class of special
 Hom 3-Lie algebras. We also study a special class of Hom 3-Lie-Rinehart algebras, called  split regular Hom 3-Lie-Rinehart algebras.

 Unless otherwise stated,  algebras and vector spaces are over a field $F$ of characteristic zero,  and $A$ denotes a commutative associative algebra over $F$, $Z$ is the set of integer. For a vector space $V$ and its subset $S$,  $\langle S\rangle$ denotes the subspace of $V$ spanned by $S$.

\section{Preliminaries}

{\it A 3-Lie algebra}~\cite{F} is a vector space $L$ over $F$ endowed with an $F$-linear multiplication $[\ ,\  ,\  ]: L\wedge L\wedge L\rightarrow L$ satisfying for all $x_{1},x_{2},x_{3},y_{2},y_{3}\in L,$
\begin{equation}\label{eq:jacobi}
[[x_{1},x_{2},x_{3}],y_{2},y_{3}]=[[x_{1},y_{2},y_{3}],x_{2},x_{3}]+[[x_{2},y_{2},y_{3}],x_{3},x_{1}]+[[x_{3},y_{2},y_{3}],x_{1},x_{2}].
\end{equation}

Let $L$ be a 3-Lie algebra,  $V$ be a vector space, and
$\rho: L\wedge L\rightarrow gl (V)$ be an $F$-linear mapping. If $\rho$ satisfies that for all $x_i\in L, 1\leq i\leq 4,$

\begin{equation}\label{eq:mod1}
[\rho(x_1, x_2), \rho(x_3, x_4)]=\rho([x_1, x_2, x_3], x_4)-\rho([x_1, x_2, x_4], x_3),
  \end{equation}

\begin{equation}
\label{eq:mod2} \rho([x_1, x_2, x_3], x_4)=\rho(x_1, x_2)\rho(x_3, x_4)+\rho(x_2, x_3)\rho(x_1, x_4)+\rho(x_3, x_1)\rho(x_2, x_4),
\end{equation}
\\
then $(V, \rho)$ is called {\it a representation} of $L$, or $(V, \rho)$ is {\it an $L$-module}. The subspace $Ker\rho=\{x\in L~~~|~~~\rho(x,L)=0\}$ is called {\it the kernel of the representation \cite{Ka}.}

Thanks to \eqref{eq:jacobi}, $(L, \mbox{ad})$ is a representation of the 3-Lie algebra $L$, and it is called  the regular representation of $L$, where
$ \forall x, y, x_1, x_2, y_1, y_2\in L,$
\begin{equation}\label{eq:ad}
\mbox{ad}: L\wedge L\rightarrow gl(L), \mbox{ad}_{x, y}z=[x, y, z],~~~~ [\mbox{ad}_{x_1,  y_1}, \mbox{ad}_{x_2,  y_2}]=\mbox{ad}_{[x_1, y_1, x_2],  y_2}+\mbox{ad}_{x_2, [x_1, y_1, y_2]}.
\end{equation}

\begin{defn} \cite{B6}\label{defin:rinehart}
Let $A$ be an commutative associative algebra over $F$, $(A, \rho)$ be a 3-Lie algebra $L$-module with $\rho(L, L)\subseteq Der(A)$, and $L$ be an $A$-module.
If   $\rho$ satisfies that
\begin{equation}\label{eq:Rinhart1}
    [x, y, az]=a[x, y, z]+\rho(x, y)a z,~~\forall x, y, z\in L, a\in A,
\end{equation}
then  $(L, A, \rho)$ is  called {\it a weak 3-Lie-Rinehart algebra}. Furthermore,

$1)$ if $\rho$ satisfies
\begin{equation}\label{eq:Rinhart2}
 \rho(ax, y)=\rho(x,ay)=a\rho(x, y),~~\forall x, y\in L, a\in A,
 \end{equation}
 then  $(L, A, \rho)$ is  called {\it a 3-Lie-Rinehart algebra};

$2)$ if $\rho=0$, then  $(L, A)$ is called {\it a 3-Lie $A$-algebra.}
\end{defn}

\begin{exam}\label{exam:1}
Let $L=C^{\infty}(R^3)$ consist of 3-ary infinitely differentiable functions over the real field $R$, $A=L$ (as  vector spaces) be the commutative associative algebra with the product of the usual function. Then $(L,A,\rho_{\mbox{ad}})$ is a weak 3-Lie-Rinehart algebra,  but  it is not a 3-Lie-Rinehart algebra, where $\forall f, g, h\in L, a\in A$,
\begin{equation}\label{eq:af}
(af)(x,y,z)=a(x,y,z)f(x,y,z),
\end{equation}
\begin{equation}\label{eq:f}
[f,g,h]_{\partial}= \frac {\partial(f, g, h)}{\partial(x, y, z)}=\left|
                                \begin{array}{ccc}
                                  \partial_xf & \partial_yf & \partial_zf \\
                                  \partial_xg & \partial_yg & \partial_zg \\
                                  \partial_xh & \partial_yh & \partial_zh \\
                                \end{array}
                              \right|,\\
\end{equation}
$$\rho_{\mbox{ad}}:L\wedge L \rightarrow gl(A),\quad  \rho_{\mbox{ad}}(f, g)(a)=\frac {\partial(f, g, a)}{\partial(x, y, z)}\in A.$$
\end{exam}

\begin{proof}
Thanks to \cite{In}, $(L, [\ ,\ ,\ ]_{\partial})$ is a 3-Lie algebra and $L$ is an $A$-module according to the usual multiplication of functions.  Since $\forall f, g, h, d\in L, a, b\in A$,
$$
  \rho_{\mbox{ad}}(f,g)(ab)=\left|
                                \begin{array}{ccc}
                                  \partial_xf & \partial_yf & \partial_zf \\
                                  \partial_xg & \partial_yg & \partial_zg \\
                                  \partial_x(ab) & \partial_y(ab) & \partial_z(ab) \\
                                \end{array}
                              \right|=a\rho_{\mbox{ad}}(f,g)(b)+b\rho_{\mbox{ad}}(f,g)(a),$$
we get  $\rho_{\mbox{ad}}(L, L) \subseteq Der(A)$, and
\begin{equation*}
  \begin{split}
  &\rho_{\mbox{ad}}([f,g,h]_\partial,d)(a)=[[f,d,a]_\partial,g,h]_\partial+[f,[g,d,a]_\partial,h]_\partial+[f,g,[h,d,a]_\partial]_\partial\\
  =&\rho_{\mbox{ad}}(f,g)\rho_{\mbox{ad}}(h,d)(a)+\rho_{\mbox{ad}}(g,h)\rho_{\mbox{ad}}(f,d)(a)+\rho_{\mbox{ad}}(h,f)\rho_{ad}(g,d)(a),
  \end{split}
\end{equation*}

\begin{equation*}
  \begin{split}
  &[\rho_{\mbox{ad}}(f,g),\rho_{\mbox{ad}}(h,d)](a)=[f,g,[h,d,a]_\partial]_\partial-[h,d,[f,g,a]_\partial]_\partial\\
    =&\rho_{\mbox{ad}}([f,g,h]_\partial,d)(a)+\rho_{\mbox{ad}}(h,[f,g,a]_\partial).
  \end{split}
\end{equation*}
It follows that $(A,\rho_{\mbox{\mbox{ad}}})$ is a 3-Lie algebra $L$-module.
From \eqref{eq:f},
\begin{equation*}
  \begin{split}
[f,g,ah]_\partial=&\left|
                                \begin{array}{ccc}
                                  \partial_xf & \partial_yf & \partial_zf \\
                                  \partial_xg & \partial_yg & \partial_zg \\
                                  \partial_x(ah) & \partial_y(ah) & \partial_z(ah) \\
                                \end{array}
                              \right|=a[f,g,h]_\partial+[f,g,a]_\partial (h)=a[f,g,h]_\partial+\rho_{\mbox{ad}}(f,g)(a)h.
\end{split}
\end{equation*}
Therefore, $(L,A,\rho_{\mbox{ad}})$ is a weak 3-Lie-Rinehart algebra.

Let $a=x\in A$, $f=x, g=y \in L$. Then
$$\rho(af, g)h=2x\partial_z h\neq x\partial_z h=a\rho(f, g)h, ~~~~\forall h\in L.$$
Therefore, $(L,A,\rho_{\mbox{ad}})$ is not a 3-Lie-Rinehart algebra.
\end{proof}

\begin{exam}\label{exam:2}
 Let $B= C^{\infty}(R)$ be the commutative associative algebra over the real field $R$, and
  $$T=\langle xf(z),yf(z)~~~|~~~\forall f(z)\in C^\infty(R)\rangle \subseteq C^{\infty}(R^3).$$
 Then $(T, [ , , ]_{\partial})$ is a subalgebra of $L= C^{\infty}(R^3)$ in Example \ref{exam:1} and $T$ is a $B$-module according to the usual multiplication of functions. Since  $\forall u,v,w\in T, a, b\in B, $ $\rho_{\mbox{ad}}(T, T)\subseteq Der(B),$ and
 \begin{equation*}
  \rho_{\mbox{ad}}(au,v)(b)=a[u,v,b]_\partial+u[a,v,b]_\partial=a\rho_{\mbox{ad}}(u,v)(b)=\rho_{\mbox{ad}}(u,av)(b).
\end{equation*}
 We get that $(T, B, \rho_{\mbox{ad}})$ is a 3-Lie-Rinehart algebra.
\end{exam}

\begin{defn} \cite{P} \label{defin:hom}
 A {\it Hom 3-Lie algbra} is a triple $(L,[~,~,~]_L,\alpha)$ over $F$ consisting of a vector space $L$, a linear  multiplication $[~,~,~]_L:\wedge^3L\rightarrow L$ and an
$F$-linear map $\alpha:L\rightarrow L$,  satisfying
\begin{equation}\label{eq:h1}
\begin{split}
    [\alpha(x_1), \alpha(x_2), [x_3,x_4,x_5]_L]_L&=[[x_1,x_2,x_3]_L,\alpha(x_4),\alpha(x_5)]_L+[\alpha(x_3), [x_1,x_2,x_4]_L, \alpha(x_5)]_L\\
    &+[\alpha(x_3),\alpha(x_4),[x_1,x_2,x_5]_L]_L, ~~ \forall x_i\in L, 1\leq i\leq 5.
\end{split}
\end{equation}
The identity \eqref{eq:h1} is called the Hom-Jacobi identity. Further,

$1)$ if $\alpha$ is an algebra homomorphism, that is,  $\alpha([x,y,z]_L)=[\alpha(x),\alpha(y),\alpha(z)]_L$, $\forall x, y, z\in L$,  then $(L,[~,~,~]_L,\alpha)$ is called  a {\it multiplicative Hom 3-Lie algebra};

$2)$ if $\alpha$ is an algebra automorphism, then $(L,[~,~,~]_L,\alpha)$ is called a {\it regular Hom 3-Lie algebra}.
\end{defn}

 {\it A  sub-algebra (an ideal)} of  regular Hom 3-Lie algebra $(L,[~,~,~]_L,\alpha)$ is a subspace $I$ of $L$ satisfying  $[I,I,I]_L\subseteq I$ ~~(~~$[I, L, L]_L\subseteq I$~~) and  $\alpha(I)\subseteq I$.

 {\it The center} of the   regular Hom 3-Lie sub-algebra $(L,[~,~,~]_L,\alpha)$ is
\begin{equation}\label{eq:decenter}
Z(L)=\{ x ~~ |~~ x\in L, [x, L, L]_L=0~~\}.
\end{equation}

 It is clear that $Z(L)$ is an ideal.

\begin{exam}\label{exam:3}
Let
$L_1=\langle h(x, y)e^{kz}~|~ h(x,y)\in C^\infty(R^2),  k\in Z\rangle$ be the subspace of $L= C^{\infty}(R^3)$. Then $(L_1,[~,~,~]_\partial, \alpha_1)$ is a regular Hom 3-Lie algebra, where
$\alpha_1=-I_{d_{L_1}}: L_1\rightarrow L_1.$

In fact, from  \eqref {eq:f} and a direct computation, we have

\vspace{2mm}$
[fe^{k_1z},ge^{k_2z},he^{k_3z}]_\partial=
\left|
                                \begin{array}{ccc}
                                  \partial_xf & \partial_yf & k_1f\\
                                  \partial_xg & \partial_yg & k_2g\\
                                  \partial_xh& \partial_yh & k_3h \\
                                \end{array}
                              \right|e^{(k_1+k_2+k_3)z}\in L_1,
                              $

\vspace{2mm}$[\alpha_1(fe^{k_1z}),\alpha_1(ge^{k_2z}),\alpha_1(he^{k_3z})]_\partial=\alpha_1[fe^{k_1z},ge^{k_2z},he^{k_3z}]_\partial,$

\begin{equation*}
  \begin{split}
  &[\alpha_1(f_1e^{k_1z}),\alpha_1(f_2e^{k_2z}),[f_3e^{k_3z},f_4e^{k_4z},f_5e^{k_5z}]_\partial]_\partial\\
  =&[[f_1e^{k_1z},f_2e^{k_2z},f_3e^{k_3z}]_\partial,\alpha_1(f_4e^{k_4z}),\alpha_1(f_5e^{k_5z})]_\partial+[\alpha_1(f_3e^{k_3z}),[f_1e^{k_1z},f_2e^{k_2z},f_4e^{k_4z}]_\partial,\alpha_1(f_4e^{k_4z})]_\partial\\
  +&[\alpha_1(f_3e^{k_3z}),\alpha_1(f_4e^{k_4z}),[f_1e^{k_1z},f_2e^{k_2z},f_5e^{k_5z}]_\partial]_\partial,~~ \forall fe^{k_1z}, ge^{k_2z}, he^{k_3z}\in L_1, f_ie^{k_iz}\in L_1, 1\leq i\leq 5.
  \end{split}
  \end{equation*}
Therefore,  $(L_1,[~,~,~]_\partial, -Id_{L_1})$ is a regular Hom 3-Lie algebra.
\end{exam}

\begin{defn} \cite{P} \label{defin:homr}
A {\it representation of multiplicative Hom 3-Lie algebra} $(L,[~,~,~]_L,\alpha)$ over $F$ is a triple $(V,\phi,\rho)$, where $V$ is a vector space over $F$, $\rho: L\wedge L\rightarrow gl(V)$ and $\phi: V\rightarrow V$ are $F$-linear mappings satisfy that
 \begin{equation}\label{eq:hr1}
    \rho(\alpha(x_1),\alpha(x_2))\phi=\phi\rho(x_1,x_2),
\end{equation}
 \begin{equation}\label{eq:hr2}
\begin{split}
    \rho([x_1,x_2,x_3]_L,\alpha(x_4))\phi=&\rho(\alpha(x_1),\alpha(x_2))\rho(x_3,x_4)+\rho(\alpha(x_2),\alpha(x_3))\rho(x_1,x_4)\\
    +&\rho(\alpha(x_3),\alpha(x_1))\rho(x_2,x_4),
\end{split}
\end{equation}
 \begin{equation}\label{eq:hr3}
\begin{split}
    \rho(\alpha(x_1),\alpha(x_2))\rho(x_3,x_4)=&\rho(\alpha(x_3),\alpha(x_4))\rho(x_1,x_2)+\rho([x_1,x_2,x_3]_L,\alpha(x_4))\phi\\
    +&\rho(\alpha(x_3),[x_1,x_2,x_4]_L)\phi,  \forall x_i\in L, 1\leq i\leq 4.
\end{split}
\end{equation}
\end{defn}

It is clear that if $\phi=I_{d_V}$, then $(V,\rho)$ is  a representation of Hom 3-Lie algebra $(L,[~,~,~]_L,\alpha)$  defined in \cite{B7}.

\begin{remark}\label{rem:r1}
Thanks to  \eqref{eq:hr2}, \eqref{eq:hr3} is equivalent to,  $\forall x_i\in L$, $1\leq i\leq 4$,
 \begin{equation}\label{eq:hr4}
\begin{split}
0=&\rho(\alpha(x_1),\alpha(x_2))\rho(x_3,x_4)+\rho(\alpha(x_2),\alpha(x_3))\rho(x_1,x_4)\rho(\alpha(x_3),\alpha(x_1))\rho(x_2,x_4)\\
+&\rho(\alpha(x_3),\alpha(x_4))\rho(x_1,x_2)+\rho(\alpha(x_1),\alpha(x_4))\rho(x_2,x_3)+\rho(\alpha(x_2),\alpha(x_4))\rho(x_3,x_1)).
\end{split}
\end{equation}
\end{remark}

\begin{defn} \cite{H} \label{defin:hd}
Let $\phi:A\rightarrow A$ be an algebra endomorphism. A $\phi$-derivation on $A$ is an $F$-linear map $D: A\rightarrow A$ satisfying
\begin{equation}\label{eq:hd1}
    D(ab)=\phi(a)D(b)+D(a)\phi(b), ~~ \forall a,b\in A.
\end{equation}
 $Der_\phi(A)$ denotes the set of {\it $\phi$-derivations} on $A$.
\end{defn}

Since $\phi$ is an algebra endomorphism,  we have
\begin{equation}\label{eq:hd2}
    D(abc)=\phi(ab)D(c)+\phi(bc)D(a)+\phi(ac)D(b), ~~ \forall a,b,c\in A.
\end{equation}

\section{Basic structures of hom 3-Lie-Rinehart algenras}
\mlabel{sec:rbl3rbl}

\begin{defn}\label{defin:hl} Let $(L,[~,~,~]_L,\alpha)$ be a multiplicative Hom 3-Lie algebra, $\phi: A\rightarrow A$ be an algebra homomorphism and
$\rho: L\wedge L\rightarrow Der_\phi(A)$  be an $F$-linear map. If  $(A,\rho,\phi)$ is a representation of the Hom 3-Lie algebra $(L,[~,~,~]_L,\alpha)$, and $\forall a\in A, x, y, z\in L,$
\begin{equation}\label{eq:hom1}
\alpha(ax)=\phi(a)\alpha(x), \quad [x,y,az]_L=\phi(a)[x,y,z]_L+\rho(x,y)(a)\alpha(z),
 \end{equation}
then the tuple $(L,A,[~,~,~]_L,\phi,\alpha,\rho)$ is called  {\it a weak Hom 3-Lie-Rinehart algebra} over $(A,\phi)$.

Further, if $\rho$ satisfies that
 \begin{equation}\label{eq:amod}
 \rho(ax,y)=\rho(x,ay)=\phi(a)\rho(x,y), \forall a\in A, x,y\in L,
 \end{equation}
then $(L,A,[~,~,~]_L,\phi,\alpha,\rho)$ is called a {\it Hom 3-Lie-Rinehart algebra}.

If $\alpha$ and $\phi$ are algebra automorphisms, then $(L,A,[~,~,~]_L,\phi,\alpha,\rho)$ is called  {\it a regular  Hom 3-Lie-Rinehart algebra}.
\end{defn}

From Definition \ref{defin:hl}, we know that
 every 3-Lie-Rinehart algebra $(L, A,\rho)$ is a Hom 3-Lie-Rinehart algebra $(L,A,[~,~,~]_L,$ $ I_{d_A}, $ $I_{d_L}, \rho)$; and
every multiplicative Hom 3-Lie algebra $(L,$ $[~,~,~]_L,$ $\alpha)$ is a Hom 3-Lie-Rinehart algebra $(L,F, [~,~,~]_L, I_{d_F},$ $\alpha,0)$.

A Hom 3-Lie-Rinehart algebra  $(L,A,[~,~,~]_L,\phi,\alpha,\rho)$ is usually denoted  by $(L,A,\phi,\alpha,\rho)$ or $L$, if there is no confusion.

\begin{defn}\label{defn:subandideal} Let $(L,A,[~,~,~]_L,\phi,\alpha,\rho)$ be a Hom 3-Lie-Rinehart algebra.

  1)  If $S$ is a subalgebra of the Hom 3-Lie algebra $(L, [~,~,~]_L,\alpha)$ satisfying $AS\subset S$, then   $(S, A,\phi, \alpha|_S,$ $ \rho|_{S\wedge S})$ is a Hom 3-Lie-Rinehart algebra, which is called {\it a subalgebra of the Hom 3-Lie-Rinehart algebra $(L,A,\phi,\alpha,\rho)$.}

2) If $I$ is an ideal of the Hom 3-Lie algebra $(L,[~,~,~]_L, \alpha)$ and satisfies $AI\subset I$ and $\rho(I,L)(A)L\subset I$, then   $(I, A, \phi, \alpha|_I, \rho|_{I\wedge I})$ is a Hom 3-Lie-Rinehart algebra, which is called {\it an ideal of the Hom 3-Lie-Rinehart algebra $(L,A,\phi,\alpha,\rho)$.}

3) If a Hom 3-Lie-Rinehart algebra $(L,A,\phi,\alpha,\rho)$  cannot be decomposed into the direct sum of two nonzero ideals, then $L$ is called {\it an indecomposable Hom 3-Lie-Rinehart algebra}.
\end{defn}

\begin{prop}\label{prop:ker} If  $(L,A,\phi,\alpha,\rho)$ is  a regular Hom 3-Lie-Rinehart algebra, then $Ker\rho$ is an ideal.
\end{prop}

\begin{proof}
 Thank to \eqref{eq:hr1} and \eqref{eq:hr2}, for $\forall x\in Ker\rho$, $y,z,w\in L$,
 $$
    \rho([x,y,z]_L,\alpha(w))\phi=\rho(\alpha(x),\alpha(y))\rho(z,w)+\rho(\alpha(y),\alpha(z))\rho(x,w)+\rho(\alpha(z),\alpha(x))\rho(y,w)=0,
$$
$$
\rho(\alpha(x),\alpha(y))\phi=\phi\rho(x,y)=0.
$$
Therefore, $[x, L, L]_L\subseteq Ker\rho,$ and  $\alpha(Ker\rho)\subset Ker\rho$.
By \eqref{eq:amod}, for $\forall a\in A$,
$$\rho(ax,y)=\rho(x,ay)=\phi(a)\rho(x,y)=0, \quad\rho(\rho(x,y)az,w)=\phi(\rho(x,y)a)\rho(z,w)=0.$$
We get  $ax\in Ker\rho$, $\rho(x,y)az\in Ker\rho$, that is, $A Ker\rho\subset Ker\rho$, $\rho(Ker\rho,L)(A)L\subset Ker\rho$.

Therefore, $Ker\rho$ is an ideal of the regular Hom 3-Lie-Rinehart algebra $(L,A,\phi,\alpha,\rho)$.
\end{proof}

\begin{exam}\label{exam:4}
 Let $(L,A, [ \ ,\ ,\ ]_{\partial}, \rho_{\mbox{ad}})$ be the weak 3-Lie-Rinehart in Example \ref{exam:1}.  Define  $R$-linear maps
  $$  \alpha=-I_{d_L}: L\rightarrow L,~~\alpha(f)=-f,\quad
     \phi=I_{d_A}: A\rightarrow A,~~\phi(a)=a,
   $$
  \begin{equation}\label{eq:ex3}
\rho'=-\rho_{\mbox{ad}}: L\wedge L\rightarrow gl(A),~~\rho'(f,g)(a)=-\frac {\partial(f, g, a)}{\partial(x, y, z)}, \forall f, g\in L, a\in A.
\end{equation}
Then $(L, A, I_{d_A}, -I_{d_{L}}, \rho')$ is a regular weak Hom  3-Lie-Rinehart algebra, and $Ker\rho'=\langle 1\rangle$.
\end{exam}

\begin{proof} It is clear that $ \alpha=-I_{d_L}$ and $ \phi=I_{d_A}$ are algebra isomorphism, and  $\forall a\in A,f\in L$, $\alpha(af)=-af=\phi(a)\alpha(f).
$ Thanks to  \eqref{eq:ex3}, $\forall f,g\in L, a,b\in A,$
$ \rho'(f,g)(ab)=\phi(a)\rho'(f,g)(b)+\rho'(f,g)(a)\phi(b),$
it follows $\rho'(L, L)\subseteq Der_\phi(A)$.

Since $(L, [\ ,\ ,\ ]_{\partial})$ is a 3-Lie algebra, $\forall f_i\in L, 1\leq i\leq 4, a\in A$,

\vspace{2mm}$
 \rho'(\alpha(f_1),\alpha(f_2))\phi(a)=-[f_1,f_2,a]_\partial=\phi\rho'(f_1,f_2)(a),
$

\vspace{2mm}$\rho'([f_1,f_2,f_3]_\partial,\alpha(f_4))\phi(a)$

\vspace{2mm}\noindent $=\rho'(\alpha(f_1),\alpha(f_2))\rho'(f_3,f_4)(a)$
$+\rho'(\alpha(f_2),\alpha(f_3))\rho'(f_1,f_4)(a)+\rho'(\alpha(f_3),\alpha(f_1))\rho'(f_2,f_4)(a),
 $

 \vspace{2mm} $\rho'(\alpha(f_1),\alpha(f_2))\rho'(f_3,f_4)(a)-\rho'(\alpha(f_3),\alpha(f_4))\rho'(f_1,f_2)(a)
 $

 \vspace{2mm}\noindent $=\rho'([f_1,f_2,f_3]_\partial,\alpha(f_4))\phi(a)+\rho'(\alpha(f_3),[f_1,f_2,f_4]_\partial)\phi(a).
 $
\\It follows that $(A,\rho',\phi)$ is a representation of $(L,[~,~,~]_\partial,\alpha)$, and for $f\in L$, $\rho'(f, L)=0$ if and only if $f$ is constant.

Thanks to  \eqref{eq:f}, $\forall a\in A, f,g,h\in L$,
$$
[f,g,ah]_\partial=a[f,g,h]_\partial+[f,g,h]_\partial a=\phi(a)[f,g,h]_\partial+\rho'(f,g)(a)\alpha(h).
$$

Therefore, $(L, A,[\ ,\ ,\ ]_{\partial}, I_{d_A}, -I_{d_{L}}, \rho')$ is a  regular weak Hom 3-Lie-Rinehart algebra with  $Ker\rho'=\langle 1\rangle$.
\end{proof}

\begin{exam}\label{exam:5} $(L_1,A_1,[~,~,~]_\partial, I_{d_{A_1}}, -I_{d_{L_1}}, -\rho_{\mbox{ad}})$ is
a weak Hom 3-Lie-Rinehart algebras, where $(L_1,[~,~,~]_\partial, -I_{d_{L_1}})$ is a regular Hom 3-Lie algebra in Example \ref{exam:3}, and $A_1=L_1$ (as vector spaces);
 $(T,B,[~,~,~]_\partial, I_{d_{B}}, -I_{d_{T}},\rho'=\rho_{-\mbox{ad}})$ is a Hom 3-Lie-Rinehart algebras, where
$B$ and $T$ are in Example \ref{exam:2}.
\end{exam}

\begin{theorem}
Let  $(L,A,[~,~,~]_L,\phi,\alpha,\rho)$ be a Hom 3-Lie-Rinehart algebra. Then the following identities hold,  for all $a,b,c\in A, x_i\in L, 1\leq i\leq 5,$
\begin{equation}\label{eq:ho1}
\begin{split}
0&=\phi\rho(x_4,x_5)(a)\alpha([x_1,x_2,x_3]_L)+\phi\rho(x_5,x_3)(a)\alpha([x_1,x_2,x_4]_L)\\
 &+\phi\rho(x_3,x_4)(a)\alpha([x_1,x_2,x_5]_L)+\phi\rho(x_2,x_3)(a)\alpha([x_1,x_4,x_5]_L)\\
 &+\phi\rho(x_2,x_4)(a)\alpha([x_3,x_1,x_5]_L)+\phi\rho(x_2,x_5)(a)\alpha([x_2,x_4,x_1]_L),
 \end{split}
  \end{equation}

\begin{equation}\label{eq:ho2}
\begin{split}
0&=\phi^2(b)(\phi\rho(x_4,x_5)(a)\alpha([x_1,x_2,x_3]_L)+\phi\rho(x_5,x_3)(a)\alpha([x_1,x_2,x_4]_L)\\
&+\phi\rho(x_3,x_4)(a)\alpha([x_1,x_2,x_5]_L)+\phi\rho(x_3,x_1)(a)\alpha([x_2,x_4,x_5]_L)\\
&+\phi\rho(x_4,x_1)(a)\alpha([x_3,x_2,x_5]_L)+\phi\rho(x_5,x_1)(a)\alpha([x_3,x_4,x_2]_L)),
 \end{split}
  \end{equation}

\begin{equation}\label{eq:ho3}
\begin{split}
 0&=\phi^2(b)(\phi\rho(x_2,x_3)(a)\alpha([x_1,x_4,x_5]_L)+\phi\rho(x_2,x_4)(a)\alpha([x_3,x_1,x_5]_L)\\
 &+\phi\rho(x_2,x_5)(a)\alpha([x_2,x_4,x_1]_L)+\phi\rho(x_1,x_3)(a)\alpha([x_2,x_4,x_5]_L)\\
 &+\phi\rho(x_1,x_4)(a)\alpha([x_3,x_2,x_5]_L)+\phi\rho(x_1,x_5)(a)\alpha([x_3,x_4,x_2]_L)),
 \end{split}
  \end{equation}

\begin{equation}\label{eq:ho4}
0=\phi(\rho(x_{1},x_{2})(a)\rho(x_{3},x_{4})+\rho(x_{1},x_{4})(a)\rho(x_{2},x_{3})+\rho(x_{2},x_{4})(a)\rho(x_{3},x_{1}))(b)\alpha^2(x_5),
  \end{equation}

\begin{equation}\label{eq:ho5}
0=\phi^2(c)\phi(\rho(x_{1},x_{2})(a)\rho(x_{3},x_{4})+\rho(x_{2},x_{3})(a)\rho(x_{1},x_{4})+\rho(x_{3},x_{1})(a)\rho(x_{2},x_{4}))(b)\alpha^2(x_5),
  \end{equation}

\begin{equation}\label{eq:ho6}
\begin{split}
0&=\phi^2(c)\phi(\rho(x_{1},x_{4})(a)\rho(x_{2},x_{3})+\rho(x_{2},x_{4})(a)\rho(x_{3},x_{1})+\rho(x_{2},x_{3})(a)\rho(x_{4},x_{1})\\
&+\rho(x_{3},x_{1})(a)\rho(x_{4},x_{2}))(b)\alpha^2(x_5).
 \end{split}
  \end{equation}
\end{theorem}

\begin{proof}
Eqs. \eqref{eq:ho1} - \eqref{eq:ho6} can be verified   by a direct computation according to \eqref{eq:h1} and Definition \ref{defin:hl}.
\end{proof}

\begin{theorem}
Let $(L, A,\rho)$ be a 3-Lie-Rinehart algebra, $\phi: A\rightarrow A$ and $\alpha: L\rightarrow L$ be algebra endomorphisms that  satisfy
 \begin{equation}\label{eq:ho7}
\rho(\alpha(x),\alpha(y))\phi=\phi\rho(x,y),~~ \alpha(ax)=\phi(a)\alpha(x), ~~\forall a\in A, x,y\in L.
 \end{equation}
  Then the tuple $(A,L,[~,~,~]_\alpha,\phi,\alpha,\rho_\phi)$ is a Hom 3-Lie-Rinehart algebra, where
\begin{equation}\label{eq:ho77}
[x,y,z]_\alpha:=\alpha([x,y,z]),~~ \quad \rho_\phi(x,y)(a):=\phi(\rho(x,y)(a)),  ~~ \forall x,y,z\in L, a\in A.
 \end{equation}
 \end{theorem}

\begin{proof} Since $\alpha$ is an algebra homomorphism, by \cite{P}, $(L, [\ ,\ ,\ ]_{\alpha}, \alpha)$ is   a multiplicative Hom 3-Lie algebra.
Thanks to \eqref{eq:ho77},  $\forall x_1, x_2\in L, a,b\in A$,
$$\rho_\phi(x_1,x_2)(ab)=\phi(\rho(x_1,x_2)(ab))=\rho_\phi(x_1,x_2)(a)\phi(b)+\phi(a)\rho_\phi(x_1,x_2)(b),$$
$$\rho_\phi(ax_1,x_2)(b)=\phi(a\rho(x_1,x_2)(b))=\phi(a)\phi(\rho(x_1,x_2)(b))=\phi(a)\rho_\phi(x_1,x_2)(b),$$
$$\rho_\phi(\alpha(x_1),\alpha(x_2))\phi=\phi(\rho(\alpha(x_1),\alpha(x_2))\phi)=\phi(\phi\rho(x_1,x_2))=\phi\rho_\phi(x_1,x_2).$$
Then $\rho_\phi(L, L)\subseteq Der_\phi(A)$, and
\begin{equation*}
\begin{split}
&\rho_\phi(\alpha(x_1),\alpha(x_2))\rho_\phi(x_3,x_4)+\rho_\phi(\alpha(x_2),\alpha(x_3))\rho_\phi(x_1,x_4)+\rho_\phi(\alpha(x_3),\alpha(x_1))\rho_\phi(x_2,x_4)\\
=&\phi\rho(\alpha(x_1),\alpha(x_2))(\phi\rho(x_3,x_4))+\phi\rho(\alpha(x_2),\alpha(x_3))(\phi\rho(x_1,x_4))\\
+&\phi\rho(\alpha(x_3),\alpha(x_1))(\phi\rho(x_2,x_4))=\rho_\phi([x_1,x_2,x_3]_\alpha,\alpha(x_4))\phi,
\end{split}
\end{equation*}
\begin{equation*}
\begin{split}
0=&\rho_\phi(\alpha(x_1),\alpha(x_2))\rho_\phi(x_3,x_4)+\rho_\phi(\alpha(x_2),\alpha(x_3))\rho_\phi(x_1,x_4)+\rho_\phi(\alpha(x_3),\alpha(x_1))\rho_\phi(x_2,x_4)\\
+&\rho_\phi(\alpha(x_3),\alpha(x_4))\rho_\phi(x_1,x_2)+\rho_\phi(\alpha(x_1),\alpha(x_4))\rho_\phi(x_2,x_3)+\rho_\phi(\alpha(x_2),\alpha(x_4))\rho_\phi(x_3,x_1).\\
\end{split}
\end{equation*}
It follows that  $(A,\rho_\phi,\phi)$ is a representation of $(L,[~,~,~]_\alpha,\alpha)$, and $(L,A,[~,~,~]_\alpha,\phi,\alpha,\rho_\phi)$ is a Hom 3-Lie-Rinehart algebra.
\end{proof}

\begin{theorem}
Let $(L,[~,~,~]_L,\alpha)$ be a multiplicative Hom 3-Lie algebra and $(A,\rho,\phi)$ be a representation of $(L,[~,~,~]_L,\alpha)$, where $\rho(L, L)\subseteq Der_\phi(A)$ and $\phi$ is an algebra homomorphism. Then $(G,A,[~,~,~]_G,\phi,\tilde{\alpha},\tilde{\rho})$ is a Hom 3-Lie-Rinehart algebra, where $G=A\otimes L=\{ax~|~a\in A,x\in L\}$, and $\forall x_i\in L$, $a_i\in A$, $i=1,2,3$,
\begin{equation}\label{eq:ho8}
\begin{split}
[a_1x_1,a_2x_2,a_3x_3]_G:=&\phi(a_1a_2a_3)[x_1,x_2,x_3]_L+\phi(a_1a_2)\rho(x_1,x_2)(a_3)\alpha(x_3)\\
+&\phi(a_2a_3)\rho(x_2,x_3)(a_1)\alpha(x_1)+\phi(a_1a_3)\rho(x_3,x_1)(a_2)\alpha(x_2),
\end{split}
\end{equation}
\begin{equation}\label{eq:ho9}
\tilde{\alpha}:G\rightarrow G, \quad \tilde{\alpha}(ax)=\phi(a)\alpha(x),~~\forall a\in A, x\in L,
\end{equation}
\begin{equation*}
\tilde{\rho}:G\wedge G\rightarrow gl(A), \quad \tilde{\rho}(a_1x_1,a_2x_2)=\phi(a_1a_2)\rho(x_1,x_2),~~\forall a_1,a_2,\in A, x_1,x_2\in L.
\end{equation*}
\end{theorem}

\begin{proof}
It is apparently that $(G,[~,~,~]_G,\tilde{\alpha})$ is a Hom 3-Lie algebra and $\tilde{\rho}(G, G)\subseteq Der_{\phi}(A)$.

 By Eqs. \eqref{eq:hr1}, \eqref{eq:ho8} and \eqref{eq:ho9},  $\forall x_i\in L$, $a_i\in A$, $i=1,2,3$,
\begin{equation*}
\tilde{\alpha}([a_1x_1,a_2x_2,a_3x_3]_G)=[\phi(a_1)\alpha(x_1),\phi(a_2)\alpha(x_2),\phi(a_3)\alpha(x_3)]_G=[\tilde{\alpha}(a_1x_1),\tilde{\alpha}(a_2x_2),\tilde{\alpha}(a_3x_3)]_G.
\end{equation*}
Then $(G,[~,~,~]_G,\tilde{\alpha})$ is a multiplicative Hom 3-Lie algebra. Thanks to \eqref{eq:ho4}, \eqref{eq:ho5} and Definition \ref{defin:homr}, $(A,\tilde{\rho},\phi)$ is a representation of $(G,[~,~,~]_G,\tilde{\alpha})$.

By \eqref{eq:ho8} and \eqref{eq:ho9}, $\forall a,b,b',c, a_i\in A$, $x,y, x_i\in L$, $i=1,2,3$,

\vspace{2mm}$
\tilde{\alpha}(b'(bx))=\phi(b'b)\alpha(x)=\phi(b')\tilde{\alpha}(bx),
$

\vspace{2mm}
$\tilde{\rho}(ax, b'(by))(c)=\phi(b')\phi(ab)\rho(x,y)(c)=\tilde{\rho}(b'(ax),by)(c)=\phi(b')\tilde{\rho}(ax,by)(c),$

\vspace{2mm}$
[a_1x_1,a_2x_2,b\cdot a_3x_3]_G=\phi(b)[a_1x_1,a_2x_2,a_3x_3]_G+\tilde{\rho}(a_1x_1,a_2x_2)(b)\tilde{\alpha}(a_3x_3).
$
\\
Therefore, $(G,A,[~,~,~]_G,\phi,\tilde{\alpha},\tilde{\rho})$ is a Hom 3-Lie-Rinehart algebra.
\end{proof}

\section{ Split Regular Hom 3-Lie-Rinehart algebras}

In this section, we  study a class of  Hom 3-Lie-Rinehart algebras.

\begin{defn}\label{defin:2}
Let $(L,A,[~,~,~]_L,\phi,\alpha,\rho)$ be a regular Hom 3-Lie-Rinehart algebra. If there exist a maximal abelian subalgebra $H$ of $L$ satisfying that $\alpha(H)=H$ and
 \begin{equation}\label{eq:rootdecom}
 L=H\oplus \bigoplus \limits_{\gamma\in \Gamma}L_\gamma, ~~\mbox{where}~~  \Gamma=\{\gamma\in (H\wedge H)^*_{\neq 0}~~|~~L_\gamma\neq 0\},
 \end{equation}
 \begin{equation}\label{eq:s1}
  L_\gamma=\{x\in L~|~[h_1,h_2,x]_L=\gamma(h_1,h_2)\alpha(x),\forall h_1,h_2\in H\}, ~~ L_0=H;
  \end{equation}
  \begin{equation}\label{eq:adecom}
A=A_0\oplus \bigoplus \limits_{\lambda\in \Lambda}A_\lambda, ~~\mbox{where}~~   \Lambda=\{\lambda\in (H\wedge H)^*_{\neq 0}~|~A_\lambda\neq 0\},
\end{equation}
\begin{equation}\label{eq:ss}
  A_\lambda=\{a\in A~|~\rho(h_1,h_2)(a)=\lambda(h_1,h_2)\phi(a),\forall h_1,h_2\in H\}, ~  A_0=\{a\in A~|~\rho(H,H)(a)=0\},
  \end{equation}
then $(L,A,[~,~,~]_L,\phi,\alpha,\rho)$ is called a {\it split regular Hom 3-Lie-Rinehart algebra},   and $H$ is {\it a splitting Cartan subalgebra of $L$}, $\Gamma$ and $\Lambda$ are called the {\it root system} of  $L$, and the {\it weight system} of  $A$ associated to $H$, respectively.
\end{defn}

For convenience, in the rest of the paper, the  split regular Hom 3-Lie-Rinehart algebra $(L,A,[~,~,~]_L,\phi,\alpha,\rho)$ is simply denoted by  $(L,A)$, and denote
$$\quad -\Gamma=\{-\gamma~|~\gamma\in \Gamma\}, \quad -\Lambda=\{-\lambda~|~\lambda\in \Lambda\}, \quad \pm\Gamma=\Gamma\cup -\Gamma, \quad \pm\Lambda=\Lambda\cup -\Lambda.$$

\begin{exam}\label{exam:10}
Let $T'=\langle xe^{kz},ye^{kz},1~|~~ k\in Z\rangle\subseteq C^{\infty}(R^3),$ $B'=\langle e^{kz}~|~k\in Z \rangle\subseteq C^{\infty}(R^1)$. By a similar discussion to Example \ref{exam:4},
we find $(T',B',[~,~,~]_\partial,I_{d_{B'}},-I_{d_{T'}}, \rho_{-\mbox{ad}})$ is a  regular Hom 3-Lie-Rinehart algebra.
\end{exam}
\begin{proof}
Let $H=\langle x,y,1\rangle$ be the subspace of $T'$.

Thanks to \eqref{eq:f}, $[H, H, H]_{\partial}=0$, and for $\forall h_1=m_1x+n_1y+c_1$, $h_2=m_2x+n_2y+c_2\in H$ and $\xi_1=l_1xe^{kz}\in T'$, $\xi_2=l_2ye^{kz}\in T'$,
$u=se^{kz}\in B'$, $m_i, n_i, c_i, l_i, s\in R,$ $i=1, 2$,
$$
[h_1, h_2, \xi_1]_\partial=\left|
            \begin{array}{ccc}
              m_1 & n_1 & 0 \\
              m_2 & n_2 & 0 \\
              1 & 0 & l_1kx \\
            \end{array}
          \right|e^{kz}
          =k (m_1n_2-m_2n_1)\xi_1=k (m_2n_1-m_1n_2)(-I_{d_{T'}}(\xi_1)),$$
$$
[h_1, h_2, \xi_2]_\partial=\left|
            \begin{array}{ccc}
              m_1 & n_1 & 0 \\
              m_2 & n_2 & 0 \\
              0 & 1 & l_2ky \\
            \end{array}
          \right|e^{kz}
          =k (m_1n_2-m_2n_1)\xi_2=k (m_2n_1-m_1n_2)(-I_{d_{T'}}(\xi_2)),$$
$$\rho_{-\mbox{ad}}(h_1, h_2)u=-[h_1, h_2, u]_\partial=-k (m_1n_2-m_2n_1)u=k (m_2n_1-m_1n_2)I_{d_{B'}}(u).$$
Therefore, for $\xi_i\in T'$, $i=1, 2$, $~~ ~[H, H, \xi_i]_\partial=0 ~~\mbox{iff}~~ k=0$, that is, $[H, H, \xi_i]_\partial=0$ if and only if $\xi_i\in H,$
and $$ T'=H\oplus \bigoplus \limits_{k\in Z\setminus\{0\}}T_{\gamma_k}', \quad  B'=B_{0}'\oplus \bigoplus \limits_{k\in Z\setminus \{0\}}B_{\lambda_k}',~~ B_{0}'=R,$$
where for $k\in Z_{\neq 0},$ $T_{\gamma_k}'=\langle xe^{kz}, ye^{kz}\rangle, ~~ B_{\lambda_k}'=\langle e^{kz}\rangle$.
It follows that  $H$ is a   splitting Cartan subalgebra,
the  root system and the  weight system associated to $H$ are
$$  \Gamma=\{\gamma_k\in (H\wedge H)^*_{\neq 0}~~|~~k\in Z_{\neq 0}, \gamma_k(m_1x+n_1y+c_1, m_2x+n_2y+c_2)=k(m_2n_1-m_1n_2)\},$$
$$  \Lambda=\{\lambda_k\in (H\wedge H)^*_{\neq 0}~~|~~k\in Z_{\neq 0}, \lambda_k(m_1x+n_1y+c_1, m_2x+n_2y+c_2)=k(m_2n_1-m_1n_2)\}.$$
\end{proof}

\begin{theorem}\label{thm:1}
To Definition \ref{defin:2}, for any $\gamma,\gamma_1,\gamma_2,\gamma_3\in \Gamma\cup \{0\}$ and $\lambda,\lambda_1,\lambda_2\in \Lambda\cup \{0\}$, the following assertions hold:

$1)$  if $\lambda\in \Lambda$, then for $\forall k\in Z$, $\lambda(\alpha^k, \alpha^k)\in \Lambda$, and  $\phi^k(A_\lambda)= A_{\lambda(\alpha^{-k},\alpha^{-k})}$;

$2)$  if $\gamma\in \Gamma$, then for $\forall k\in Z$, $\gamma(\alpha^k, \alpha^k)\in \Gamma$, and $\alpha^k(L_\gamma)= L_{\gamma(\alpha^{-k},\alpha^{-k})}$;

$3)$ if $[L_{\gamma_1},L_{\gamma_2},L_{\gamma_3}]_L\neq 0$, then $(\gamma_1+\gamma_2+\gamma_3)(\alpha^{-1},\alpha^{-1})\in \Gamma\cup \{0\}$ and $$[L_{\gamma_1},L_{\gamma_2},L_{\gamma_3}]_L\subset L_{(\gamma_1+\gamma_2+\gamma_3)(\alpha^{-1}, \alpha^{-1})};$$

$4)$  if $A_{\lambda_1} A_{\lambda_2}\neq 0$, then $\lambda_1+\lambda_2\in \Lambda\cup \{0\}$ and $A_{\lambda_1} A_{\lambda_2}\subset A_{\lambda_1+\lambda_2}$;

$5)$ if $A_{\lambda} L_{\gamma}\neq 0$, then $\lambda+\gamma\in \Gamma\cup \{0\}$ and $A_{\lambda} L_{\gamma}\subset L_{\lambda+\gamma}$;

$6)$ if $\rho(L_{\gamma_1},L_{\gamma_2})(A_{\lambda})\neq 0$, then $(\gamma_1+\gamma_2+\lambda)(\alpha^{-1}, \alpha^{-1})\in \Lambda\cup \{0\}$ and $$\rho(L_{\gamma_1},L_{\gamma_2})(A_{\lambda})\subset A_{(\gamma_1+\gamma_2+\lambda)(\alpha^{-1}, \alpha^{-1})};$$
where $\forall \gamma \in \Gamma\cup \{0\}$ and $\forall \lambda \in \Lambda\cup \{0\}$,  and $h_1, h_2\in H$,
$$\gamma(\alpha^{-1},\alpha^{-1})(h_1, h_2)=\gamma(\alpha^{-1}(h_1),\alpha^{-1}(h_2)),\quad \lambda(\alpha^{-1},\alpha^{-1})(h_1, h_2)=\lambda(\alpha^{-1}(h_1),\alpha^{-1}(h_2)).$$
\end{theorem}

\begin{proof} Thanks to \eqref{eq:hr1} and \eqref{eq:ss}, $\forall a\in A_\lambda$, $h_1,h_2\in H$,
\begin{equation*}
\begin{split}
 \rho(h_1,h_2)(\phi(a))=&\rho(\alpha\alpha^{-1}(h_1),\alpha\alpha^{-1}(h_2))\phi(a) =\phi\rho(\alpha^{-1}(h_1),\alpha^{-1}(h_2))(a)\\
 =&\phi(\lambda(\alpha^{-1}(h_1),\alpha^{-1}(h_2))\phi(a)) =\lambda(\alpha^{-1}(h_1),\alpha^{-1}(h_2))\phi(\phi(a)).
\end{split}
\end{equation*}
We get $\lambda(\alpha^{-1},\alpha^{-1})\in \Lambda$, and $\phi(A_\lambda)\subset A_{\lambda(\alpha^{-1},\alpha^{-1})}$.
 By induction on $k$, suppose that
 \\
 $\phi^{k-1}(A_\lambda)\subseteq $ $ A_{\lambda(\alpha^{1-k},\alpha^{1-k})}$, that is,
 $\rho(h_1,h_2)(\phi^{k-1}(a))$ $=\lambda(\alpha^{1-k}(h_1),$ $\alpha^{1-k}(h_2))\phi(\phi^{k-1}(a)).$
Then
\begin{equation*}
\begin{split}
 \rho(h_1,h_2)(\phi^k(a))=&\rho(\alpha\alpha^{-1}(h_1),\alpha\alpha^{-1}(h_2))\phi(\phi^{k-1}(a)) =\phi(\rho(\alpha^{-1}(h_1),\alpha^{-1}(h_2))\phi^{k-1}(a))\\
 =&\phi(\lambda(\alpha^{-k}(h_1),\alpha^{-k}(h_2))\phi^k(a)) =\lambda(\alpha^{-k},\alpha^{-k})(h_1,h_2)\phi(\phi^k(a)).
\end{split}
\end{equation*}
Therefore $\phi^k(A_\lambda)\subseteq A_{\lambda(\alpha^{-k},\alpha^{-k})}$, for $k\in Z$.

Similarly,  by \eqref{eq:hr1} and induction on $k$,  we have
$$\rho(\alpha^k(h_1),\alpha^k(h_2))\phi=\phi^k\rho(h_1,h_2)\phi^{-k+1},\forall k\in Z, h_1,h_2\in H.$$
Thanks to  \eqref{eq:ss}, $\forall h_1,h_2\in H$, $a\in A_{\lambda(\alpha^{-k},\alpha^{-k})}$,  we have
\begin{equation}\label{eq:hh}
\rho(h_1,h_2)\phi^{-k}=\phi^{-k} \rho(\alpha^k(h_1),\alpha^k(h_2)),
\rho(h_1,h_2)(a)=\lambda(\alpha^{-k}(h_1),\alpha^{-k}(h_2))\phi(a).
\end{equation}
Suppose  $\phi^k(b)=a$ for some $b\in A$. By \eqref{eq:hh},
\begin{equation*}
\begin{split}
 \rho(h_1,h_2)(b)=\rho(h_1,h_2)\phi^{-k}(a)=\phi^{-k} \rho(\alpha^k(h_1),\alpha^k(h_2))(a)=\phi^{-k} (\lambda(h_1,h_2)\phi(a))=\lambda(h_1,h_2)\phi(b).
\end{split}
\end{equation*}
It follows $b\in A_\lambda$, and $A_{\lambda(\alpha^{-k},\alpha^{-k})}\subseteq  \phi^k(A_\lambda)$, for $\forall k\in Z$. The result $1)$ follows.

Thanks to Definition \ref{defin:2}, for  $\forall x\in L_\gamma$, $h_1,h_2\in H,$  $k\in Z$,
$$ [h_1,h_2,\alpha^k(x)]_L=\alpha^k(\gamma(\alpha^{-k}(h_1),\alpha^{-k}(h_2))\alpha(x))=\gamma(\alpha^{-k},\alpha^{-k})(h_1,h_2)\alpha(\alpha^k(x)).\\
$$
Therefore,  $\gamma(\alpha^k, \alpha^k)\in \Gamma$, and $\alpha^k(L_\gamma)\subseteq L_{\gamma(\alpha^{-k},\alpha^{-k})}$, for $\forall k\in Z$.

Conversely, for  $\forall y\in L_{\gamma(\alpha^{-k},\alpha^{-k})}$, we have $[h_1,h_2,y]_L=\gamma(\alpha^{-k}(h_1),\alpha^{-k}(h_2))\alpha(y)$.
 Since $\alpha$ is an algebra automorphism, there is $w\in L$ such that  $\alpha^k(w)=y$. Therefore, $\forall h_1,h_2\in H$, $k\in Z$,
$$[h_1,h_2,w]_L=[h_1,h_2,\alpha^{-k}(y)]_L=\alpha^{-k}(\gamma(h_1,h_2)\alpha(y))=\gamma(h_1,h_2)\alpha(w).
$$
It follows that $w\in L_\gamma$, and $L_{\gamma(\alpha^{-k},\alpha^{-k})}\subseteq \alpha^k(L_\gamma)$, for $\forall k\in Z$. We get $2).$

 Thanks to Definition \ref{defin:2} and Eqs.\eqref{eq:h1} and \eqref{eq:s1}, $\forall x\in L_{\gamma_1}$, $y\in L_{\gamma_2}$, $z\in L_{\gamma_3}$, $h_1,h_2\in H$,
\begin{equation*}
\begin{split}
&[h_1,h_2,[x,y,z]_L]_L=[\alpha(\alpha^{-1}(h_1)),\alpha(\alpha^{-1}(h_2)),[x,y,z]_L]_L\\
=&[\gamma_1(\alpha^{-1}(h_1),\alpha^{-1}(h_2))\alpha(x),\alpha(y),\alpha(z)]_L+[\alpha(x),\gamma_2(\alpha^{-1}(h_1),\alpha^{-1}(h_2))\alpha(y),\alpha(z)]_L\\
+&[\alpha(x),\alpha(y),\gamma_3(\alpha^{-1}(h_1),\alpha^{-1}(h_2))\alpha(z)]_L=(\gamma_1+\gamma_2+\gamma_3)(\alpha^{-1},\alpha^{-1})(h_1,h_2)\alpha([x,y,z]_L).
\end{split}
\end{equation*}
Therefore, $(\gamma_1+\gamma_2+\gamma_3)(\alpha^{-1}, \alpha^{-1})\in \Gamma,$ and $[x,y,z]_L\in L_{(\gamma_1+\gamma_2+\gamma_3)(\alpha^{-1}, \alpha^{-1})}$. We get $3).$

For $\forall a\in A_{\lambda_1}$, $b\in A_{\lambda_2}$, since    $\rho(h_1,h_2)\in Der_{\phi}(A)$, for $\forall h_1,h_2\in H$,  we have
\begin{equation*}
\begin{split}
&\rho(h_1,h_2)(ab)=\phi(a)\rho(h_1,h_2)(b)+\rho(h_1,h_2)(a)\phi(b)\\
=&\phi(a)\lambda_1(h_1,h_2)\phi(b)+\lambda_2(h_1,h_2)\phi(a)\phi(b)=(\lambda_1+\lambda_2)(h_1,h_2)\phi(ab).
\end{split}
\end{equation*}
Thus $\lambda_1+\lambda_2\in \Lambda$, and $ab\in A_{\lambda_1+\lambda_2}$. We get $4).$

 Thanks to Definition \ref{defin:hl}, $\forall a\in A_\lambda$, $x\in L_\gamma$ and $h_1,h_2\in H$,

$
[h_1,h_2,ax]_L=$ $\phi(a)[h_1,h_2,x]_L+$  $\rho(h_1,h_2)(a)\alpha(x)$ $=\phi(a)\gamma(h_1,h_2)\alpha(x)$ $+\lambda(h_1,h_2)\phi(a)\alpha(x)
$
\\$=(\lambda+\gamma)(h_1,h_2)\alpha(ax)$.\quad Therefore, $\lambda+\gamma\in \Gamma$ and $ax\in L_{\lambda+\gamma}$. We get $5)$.

 Since $(A,\rho,\phi)$ is a representation of $(L,[~,~,~]_L,\alpha)$,  $\forall a\in A_\lambda$, $x\in L_{\gamma_1}$, $y\in L_{\gamma_2}$ and $h_1,h_2\in H$, we have
\begin{equation*}
\begin{split}
&\rho(h_1,h_2)(\rho(x,y)(a))=\rho(\alpha(\alpha^{-1}(h_1)),\alpha(\alpha^{-1}(h_2))\rho(x,y)(a)\\
=&\rho(\alpha(x),\alpha(y))\lambda(\alpha^{-1}(h_1),\alpha^{-1}(h_2))\phi(a)+\rho(\gamma_1(\alpha^{-1}(h_1),\alpha^{-1}(h_2))\alpha(x),\alpha(y))\phi(a)\\
+&\rho(\alpha(x),\gamma_2(\alpha^{-1}(h_1),\alpha^{-1}(h_2))\alpha(y))\phi(a)\\
=&(\gamma_1+\gamma_2+\lambda)(\alpha^{-1},\alpha^{-1})(h_1,h_2)\phi(\rho(x,y)(a)).
\end{split}
\end{equation*}
Therefore, $(\gamma_1+\gamma_2+\lambda)(\alpha^{-1},\alpha^{-1})\in \Gamma$, and $\rho(x,y)(a)\in A_{(\gamma_1+\gamma_2+\lambda)(\alpha^{-1},\alpha^{-1})}$. We get $6).$
\end{proof}

\begin{defn}\label{defin:3}
For $\gamma,\gamma'\in \Gamma$, we say that $\gamma$ {\it is connected to} $\gamma'$ (  abbreviated by $\gamma \sim \gamma'$  ) if either $\gamma'=\pm\gamma(\alpha^k,\alpha^k)$ for some $k\in Z$, or there exists $\{\gamma_1,\gamma_2,\cdots,\gamma_{2n+1}\}\subset \pm\Gamma\cup \pm\Lambda\cup \{0\}$, with $n\geq 1$, such that

  1)~$\gamma_1\in \{\gamma(\alpha^k, \alpha^k)~|~k\in Z\}$,

  2)~$\bar{\gamma}_i=\gamma_1(\alpha^{-i}, \alpha^{-i})+\sum\limits_{j=1}^i(\gamma_{2j}+\gamma_{2j+1})(\alpha^{-i-1+j}, \alpha^{-i-1+j}) \in \pm\Gamma, ~~ 1\leq i\leq n-1,$

  3)~$\bar{\gamma}_n=\gamma_1(\alpha^{-n}, \alpha^{-n})+\sum\limits_{j=1}^n(\gamma_{2j}+\gamma_{2j+1})(\alpha^{-i-1+j}, \alpha^{-i-1+j})
\in \{\pm\gamma'(\alpha^{k}, \alpha^{k})~|~k\in Z\}.$\\
 We also say that $\{\gamma_1,\gamma_2,\cdots,\gamma_n\}$ is a connection from $\gamma$ to $\gamma'$.
\end{defn}

\begin{prop}\label{pr:1}
The relation $\sim$ in $\Gamma$ defined by Definition \ref{defin:3} is an equivalence relation.
\end{prop}
\begin{proof}
First,  for $\forall \gamma\in \Gamma$, it is clear $\gamma\sim \gamma$.

For $\gamma, \gamma', \gamma''\in \Gamma $, if  $\gamma\sim \gamma'$ and $\gamma'\sim \gamma''$.
Then there are  connections  $\{\gamma_1,\gamma_2,\cdots,\gamma_{2n+1}\}$ and $\{\beta_1,\beta_2,\cdots,\beta_{2m+1}\}$ contained in $ \pm\Gamma\cup \pm\Lambda\cup \{0\}$  from $\gamma$ to $\gamma'$, and from $\gamma'$ to $\gamma''$, respectively.

If $m\geq 1$, then by Definition \ref{defin:3},  $\{\gamma_1,\cdots,\gamma_{2n+1}, -\bar{\gamma}_n, \beta_1+\beta_2+\beta_3, \beta_4,\cdots,\beta_{2m+1} \}$ is a connection from $\gamma$ to $\gamma''$.

 If $m=0$, then by  $\gamma''\in \{\pm\gamma'(\alpha^{k}, \alpha^{k})~|~k\in Z\}$, $\{\gamma_1,\gamma_2,\cdots,\gamma_{2n+1}\}$ is a connection from $\gamma$ to $\gamma''$, therefore  $\gamma\sim \gamma''$.

Lastly,  we prove that if $\gamma\sim \gamma'$, then $\gamma'\sim \gamma$.

 Suppose $\{\gamma_1,\gamma_2,\cdots,\gamma_{2n+1}\}$ is a connection from $\gamma$ to $\gamma'$, then by Definition \ref{defin:3} and Theorem \ref{thm:1}-2),  we have $\gamma'\sim \bar{\gamma}_n$.

Since $\bar{\gamma}_n\sim \bar{\gamma}_n(\alpha^2, \alpha^2)$, and $\{\bar{\gamma}_n(\alpha^2, \alpha^2), -\gamma_{2n}(\alpha, \alpha),-\gamma_{2n+1}(\alpha, \alpha)\}$ is a connection from $\bar{\gamma}_n(\alpha^2, \alpha^2)$ to $\bar{\gamma}_{n-1}$,  we have $\bar{\gamma}_n\sim \bar{\gamma}_{n-1}$, therefore,   $\bar{\gamma}_n\sim \bar{\gamma_1}$.

By the  completely similar discussion to the above, since $\bar{\gamma}_1\sim \bar{\gamma}_1(\alpha^2, \alpha^2)$, and $\{~~\bar{\gamma}_1(\alpha^2,$ $ \alpha^2), $ $-\gamma_{3}(\alpha, \alpha), $ $-\gamma_{2}(\alpha, \alpha)~~\}$ is a connection from $\bar{\gamma}_1(\alpha^2, \alpha^2)$ to $\gamma_1$,  we get $\gamma_1\sim \gamma$. Therefore,  $\gamma'\sim \gamma$. The proof is complete.
\end{proof}

\begin{prop}\label{pr:2}
To Definition \ref{defin:3}, for $\gamma, \gamma'\in \Gamma$, the following assertions hold.

$1)$ If $\gamma\sim\gamma'$, then $\gamma(\alpha^{k_1}, \alpha^{k_1})\sim \gamma'(\alpha^{k_2}, \alpha^{k_2})$, $\forall k_1,k_2\in Z$.

$2)$ For  $\mu,\beta\in \pm\Gamma\cup \pm\Lambda\cup \{0\}$, if $\gamma\sim \gamma'$ and $\gamma+\mu+\beta\in \Gamma$, then $\gamma'\sim \gamma+\mu+\beta$.
\end{prop}

\begin{proof}
If $\gamma\sim\gamma'$, then  we need to consider two cases:

$\bullet$ If  $\gamma'=\pm\gamma(\alpha^{k_1}, \alpha^{k_1})$, then by Definition \ref{defin:3}, $\forall k_1,k_2\in Z$,
$\gamma'(\alpha^{k_2}, \alpha^{k_2})=\pm\gamma(\alpha^{k_1+k_2}, \alpha^{k_1+k_2}).$ Therefore, $\gamma(\alpha^{k_1}, \alpha^{k_1})\sim \gamma'(\alpha^{k_2}, \alpha^{k_2})$.

  $\bullet\bullet$ If there exists a connection  $\{\gamma_1,\gamma_2,\cdots,\gamma_{2n+1}\}$  from $\gamma$ to $\gamma'$, then by Definition \ref{defin:3}, $\{\gamma_1,$ $\gamma_2,$ $\cdots,$ $\gamma_{2n+1}\}$ is also a connection from $\gamma(\alpha^{k_1}, \alpha^{k_1})$ to $\gamma'(\alpha^{k_2}, \alpha^{k_2})$. It follows $1).$

The result $2)$ follows from that $\{\gamma,\mu,\beta\}$ is a  connection from $\gamma$ to $\gamma+\mu+\beta$.
\end{proof}

Thanks to Proposition \ref{pr:1}, we can consider
$$\overline{\Gamma}=\Gamma/\sim =~~\{~~[\gamma]~~~|~~~\gamma\in \Gamma~~\},
$$
where $\gamma'\in [\gamma]$ if and only if $\gamma'\in \Gamma$ and $\gamma\sim \gamma'$.

For $\gamma\in \Gamma$,  defines
\begin{equation}\label{e:1}
L_{[\gamma]}:=\bigoplus_{\xi\in [\gamma]}L_\xi, \quad
L_{0,[\gamma]}:=\Big(\sum_{\substack{{\xi\in [\gamma]}\\-\xi\in \Lambda}} A_{-\xi}L_\xi \Big)+
\Big(\sum_{\substack{
{\xi,\eta,\delta\in [\gamma]}\\
 \xi+\eta+\delta=0}} [L_\xi,L_\eta,L_\delta]_L\Big).
\end{equation}
Thanks to Theorem \ref{thm:1}, $L_{0,[\gamma]}\subset H$, and  $L_{0,[\gamma]}\cap L_{[\gamma]}=0$. Denotes
\begin{equation}\label{e:2}
I_{[\gamma]}:=L_{0,[\gamma]}\oplus L_{[\gamma]}.
\end{equation}

\begin{lemma}\label{lemma:3}
By the above notations, if $(L,A)$ is a split regular Hom 3-Lie-Rinehart algebra over $F$, then for any $[\gamma]\in \overline{\Gamma}$, the following assertions hold.

\vspace{5mm}
$1)$ ~~$[I_{[\gamma]},I_{[\gamma]},I_{[\gamma]}]_L\subset I_{[\gamma]}$, \quad $2)$ ~~$\alpha(I_{[\gamma]})\subset I_{[\gamma]}$, \quad
$3)$ ~~$AI_{[\gamma]}\subset I_{[\gamma]}$.
\end{lemma}

\begin{proof}
Since $L_{0,[\gamma]}\subset H$,  $[L_{0,[\gamma]},L_{0,[\gamma]},L_{0,[\gamma]}]_L=0$, and
\begin{equation}\label{e:3}
\begin{split}
&[I_{[\gamma]},I_{[\gamma]},I_{[\gamma]}]_L=[L_{0,[\gamma]}\oplus L_{[\gamma]},L_{0,[\gamma]}\oplus L_{[\gamma]},L_{0,[\gamma]}\oplus L_{[\gamma]}]_L\\
\subset& [L_{0,[\gamma]},L_{0,[\gamma]},L_{[\gamma]}]_L+[L_{0,[\gamma]},L_{[\gamma]},L_{[\gamma]}]_L+[L_{[\gamma]},L_{[\gamma]},L_{[\gamma]}]_L.
\end{split}
\end{equation}

Thanks to \eqref{e:1} and Theorem \ref{thm:1},  $\forall \delta\in [\gamma]$, $\delta(\alpha^{-1},\alpha^{-1})\in [\gamma]$, and  $[L_{0,[\gamma]},L_{0,[\gamma]},L_\delta]_L\subset L_{\delta(\alpha^{-1},\alpha^{-1})}$. It follows  $[L_{0,[\gamma]},L_{0,[\gamma]},L_{[\gamma]}]_L\subset L_{[\gamma]}\subset I_{[\gamma]}$.

  Next, for  $\delta,\eta\in [\gamma]$ satisfying
$[L_{0,[\gamma]},L_\delta,L_\eta]_L\neq 0$, by Theorem \ref{thm:1}, $[L_{0,[\gamma]},$ $L_\delta,L_\eta]_L\subset L_{(\delta+\eta)(\alpha^{-1},\alpha^{-1})}$.
If $\delta+\eta=0$, then $[L_{0,[\gamma]},L_\delta,L_\eta]_L\subset L_{0,[\gamma]}\subset  I_{[\gamma]}$.

  If $\delta+\eta\neq 0$,  then by Theorem \ref{thm:1},  $(\delta+\eta)(\alpha^{-1},\alpha^{-1})\in \Gamma$, and $\{\delta,\eta,0\}$ is a  connection from $\delta$ to $(\delta+\eta)(\alpha^{-1},\alpha^{-1})$. Therefore,  $[L_{0,[\gamma]},L_\delta,L_\eta]_L\subset L_{(\delta+\eta)(\alpha^{-1},\alpha^{-1})}\subset L_{[\gamma]}\subset  I_{[\gamma]}$.

Finally, we consider the item $[L_{[\gamma]},L_{[\gamma]},L_{[\gamma]}]_L$ in  \eqref{e:3}.

 Let  $\xi,\eta,\delta\in [\gamma]$ satisfy $[L_\xi,L_\eta,L_\delta]_L\neq 0$. Then  by  Theorem \ref{thm:1}, $[L_\xi,L_\eta,L_\delta]_L\subset L_{(\xi+\eta+\delta)(\alpha^{-1},\alpha^{-1})}$.
 If $\xi+\eta+\delta=0$,  then by  Theorem \ref{thm:1}, $[L_\xi,L_\eta,L_\delta]_L\subset L_{0,[\gamma]}\subset I_{[\gamma]}$.

   If $\xi+\eta+\delta\neq 0$,  then $(\xi+\eta+\delta)(\alpha^{-1},\alpha^{-1})\in \Gamma$, and $\{\xi,\eta,\delta\}$ is a  connection from $\xi$ to $(\xi+\eta+\delta)(\alpha^{-1},\alpha^{-1})$. Therefore,   $[L_\xi,L_\eta,L_\delta]_L\subset L_{(\xi+\eta+\delta)(\alpha^{-1},\alpha^{-1})}\subset L_{[\gamma]}\subset I_{[\gamma]}.$
 The result $1)$ follows.

 Thanks to Theorem \ref{thm:1} and \eqref{e:1},
\begin{equation*}
\begin{split}
&\alpha(L_{0,[\gamma]})=\Big(\sum_{\substack{{\xi\in [\gamma]}\\-\xi\in \Lambda}} \phi(A_{-\xi})\alpha(L_\xi) \Big)+
\Big(\sum_{\substack{
{\xi,\eta,\delta\in [\gamma]}\\
 \xi+\eta+\delta=0}} [\alpha(L_\xi),\alpha(L_\eta),\alpha(L_\delta)]_L\Big)\\
 =&\Big(\sum_{\substack{{\xi\in [\gamma]}\\-\xi\in \Lambda}} A_{-\xi(\alpha^{-1},\alpha^{-1})}L_{\xi(\alpha^{-1},\alpha^{-1})} \Big)+
\Big(\sum_{\substack{
{\xi,\eta,\delta\in [\gamma]}\\
 \xi+\eta+\delta=0}}[L_{\xi(\alpha^{-1},\alpha^{-1})},L_{\eta(\alpha^{-1},\alpha^{-1})},L_{\delta(\alpha^{-1},\alpha^{-1})}]_L\Big).\\
\end{split}
\end{equation*}

Therefore $\alpha(L_{0,[\gamma]})\subset L_{0,[\gamma]}$.
By  \eqref{eq:s1}, $\forall h_1,h_2\in H$, $\gamma'\in [\gamma]$, $x\in L_{\gamma'}$,
$$[h_1,h_2,\alpha(x)]_L=\alpha([\alpha^{-1}(h_1),\alpha^{-1}(h_2),x]_L)=\gamma'(\alpha^{-1}(h_1),\alpha^{-1}(h_2))\alpha(\alpha(x)),$$
that is, $\alpha(L_{[\gamma]})\subset L_{[\gamma]}$. Thanks to $\alpha$ is an isomorphism, the result $2)$ holds.

 Lastly, we prove $AI_{[\gamma]}\subset I_{[\gamma]}$.
 By  \eqref{eq:adecom}, \eqref{e:1} and \eqref{e:2}, we have
$$AI_{[\gamma]}=\Big(A_0\oplus\bigoplus_{\lambda\in \Lambda}A_\lambda\Big)\Big(\sum_{\substack{{\xi\in [\gamma]}\\-\xi\in \Lambda}} A_{-\xi}L_\xi +
\Big(\sum_{\substack{
{\xi,\eta,\delta\in [\gamma]}\\
 \xi+\eta+\delta=0}} [L_\xi,L_\eta,L_\delta]_L\oplus \bigoplus_{\xi\in [\gamma]}L_\xi\Big)\Big).$$

We  discuss it in six different cases:

 $\bullet_1$  For $\forall \xi\in [\gamma]$, if  $-\xi\in \Lambda$, then by Theorem \ref{thm:1}-$5)$ and $L$ is an $A$-module,
    we have
$$A_0(A_{-\xi}L_\xi)=(A_0A_{-\xi})L_\xi\subset A_{-\xi}L_\xi\subset L_{0,[\gamma]}.$$

  $\bullet_2$ For any $\xi,\eta,\delta\in [\gamma],$  if $\xi+\eta+\delta=0$. Then by Definition \ref{defin:hl}, $ \forall a_0\in A_0, $ $x\in L_\xi, $ $y\in L_\eta,$ $ z\in L_\delta$,  we have
$$
a_0[x,y,z]_L=[x,y,\phi^{-1}(a_0)z]_L-\rho(x,y)(\phi^{-1}(a_0))\alpha(z)\in[L_\xi,L_\eta,A_0L_\delta]_L+\rho(L_\xi,L_\eta)(A_0)\alpha(L_\delta).$$

If $A_{-\delta(\alpha^{-1},\alpha^{-1})}\neq 0$ (otherwise is trivial), then by Theorem \ref{thm:1}, $-\delta(\alpha^{-1},\alpha^{-1})\in \Lambda$, and
\\$
a_0[x,y,z]_L\in [L_\xi,L_\eta,L_\delta]_L+A_{-\delta(\alpha^{-1},\alpha^{-1})}L_{\delta(\alpha^{-1},\alpha^{-1})}
\subset L_{0,[\gamma]}.
$ Therefore,   $A_0[L_\xi,L_\eta,L_\delta]_L\subset L_{0,[\gamma]}$.

 $\bullet_3$ For any $\xi\in [\gamma]$, by  Theorem \ref{thm:1}, we have $A_0L_\xi\subset L_\xi\subset L_{[\gamma]}$.

   $\bullet_4$ For $\lambda\in \Lambda$, $\xi\in [\gamma]$, if $-\xi\in \Lambda$, we have
$A_\lambda(A_{-\xi}L_\xi)=(A_\lambda A_{-\xi})L_\xi\subset A_{\lambda-\xi}L_\xi.$

 If $\lambda-\xi\in \Lambda$ and $L_\lambda\neq 0$ (otherwise is trivial), by Proposition \ref{pr:2}-2), $\{\xi,\lambda-\xi,0\}$ is a connection from $\xi$ to $\lambda$, then $\lambda\sim \xi$, that is, $\lambda\in [\gamma]$. Therefore
 $A_\lambda(A_{-\xi}L_\xi)\subset L_\lambda\subset L_{[\gamma]}.
 $

 $\bullet_5$ For any $\xi,\eta,\delta\in [\gamma], \lambda\in \Lambda, $ and $\xi+\eta+\delta=0$, by Definition \ref{defin:hl} and Theorem \ref{thm:1},  $\forall a\in A_\lambda,$ $ x\in L_\xi,$ $ y\in L_\eta, $ $z\in L_\delta$, we have
\begin{equation*}
\begin{split}
&a[x,y,z]_L=[x,y,\phi^{-1}(a)z]_L-\rho(x,y)(\phi^{-1}(a))\alpha(z)\\
\in&[L_\xi,L_\eta,A_{\lambda(\alpha,\alpha)} L_\delta]_L+\rho(L_\xi,L_\eta)(A_{\lambda(\alpha,\alpha)})\alpha(L_\delta)\subset[L_\xi,L_\eta,L_{\lambda(\alpha,\alpha)+\delta}]_L+A_{\lambda-\delta(\alpha^{-1},\alpha^{-1})}L_{\delta(\alpha^{-1},\alpha^{-1})}.\\
\end{split}
\end{equation*}
 If $L_\lambda\neq 0$ and $L_{\lambda+\delta}\neq 0$ (otherwise is trivial), then $\lambda\in \Gamma$, $\lambda+\delta\in \Gamma$. Thanks to Proposition \ref{pr:2}-2), $\lambda\sim \delta$, $\lambda\in [\gamma]$. Therefore,
  $A_\lambda[L_\xi,L_\eta,L_\delta]_L\subset L_{[\gamma]}.$

 $\bullet_6$  For $\forall \lambda\in \Lambda$ and $\xi\in [\gamma]$, by Theorem \ref{thm:1}-$5)$ and Proposition \ref{pr:2}-2), $A_\lambda L_\xi\subset L_{\lambda+\xi}\subset L_{[\gamma]}$.

Summarizing a discussion of  six situations, we get $3)$.
\end{proof}

\begin{lemma}\label{lemma:po}
Let $[\gamma],[\eta], [\delta]\in \overline{\Gamma}$ be different from each other. Then we have
$$[I_{[\gamma]},I_{[\eta]},I_{[\delta]}]_L=0 \quad\mbox{and}\quad  [I_{[\gamma]},I_{[\gamma]},I_{[\delta]}]_L=0.$$
\end{lemma}

\begin{proof}
Thanks to \eqref{e:2},
\begin{equation}\label{eq:e1}
\begin{split}
&[I_{[\gamma]},I_{[\eta]},I_{[\delta]}]_L=[L_{0,[\gamma]}\oplus L_{[\gamma]},L_{0,[\eta]}\oplus L_{[\eta]},L_{0,[\delta]}\oplus L_{[\delta]}]_L\\
\subset &[L_{0,[\gamma]},L_{0,[\eta]},L_{[\delta]}]_L+[L_{0,[\gamma]},L_{[\eta]},L_{0,[\delta]}]_L+[L_{0,[\gamma]},L_{[\eta]},L_{[\delta]}]_L+[L_{[\gamma]},L_{0,[\eta]},L_{0,[\delta]}]_L\\
+&[L_{[\gamma]},L_{0,[\eta]},L_{[\delta]}]_L+[L_{[\gamma]},L_{[\eta]},L_{0,[\delta]}]_L
+[L_{[\gamma]},L_{[\eta]},L_{[\delta]}]_L.
\end{split}
\end{equation}

First, we consider the item $[L_{[\gamma]},L_{[\eta]},L_{[\delta]}]_L$ in \eqref{eq:e1}.

 If $[L_{[\gamma]},L_{[\eta]},L_{[\delta]}]_L\neq 0$. Then there exist $\gamma_1\in [\gamma]$, $\eta_1\in [\eta]$ and $\delta_1\in [\delta]$ such that $[L_{\gamma_1},L_{\eta_1},L_{\delta_1}]_L$ $\neq 0$. Thanks to  Theorem \ref{thm:1}, $\gamma_1+\eta_1+\delta_1\in \Gamma$. Since $\gamma\sim \gamma_1, \gamma_1+\eta_1+\delta_1\in \Gamma$ and  Proposition \ref{pr:2}-2), $\gamma\sim \gamma_1+\eta_1+\delta_1$.
 By a similar discussion to the above, we have $\eta\sim \gamma_1+\eta_1+\delta_1$. Therefore,  $\gamma\sim \eta$. Contradiction. Hence \begin{equation}\label{eq:e2}
[L_{[\gamma]},L_{[\eta]},L_{[\delta]}]_L=0.
\end{equation}

Next, we consider the item  $[L_{0,[\gamma]},L_{[\eta]},L_{[\delta]}]_L$ in \eqref{eq:e1}. Thanks to \eqref{e:1} and \eqref{e:2},
$$
[L_{0,[\gamma]},L_{[\eta]},L_{[\delta]}]_L=[\sum_{\substack{{\gamma_1\in [\gamma]}\\-\gamma_1\in \Lambda}} A_{-\gamma_1}L_{\gamma_1} +
\Big(\sum_{\substack{
{\gamma_1,\gamma_2,\gamma_3\in [\gamma]}\\
 \gamma_1+\gamma_2+\gamma_3=0}} [L_{\gamma_1},L_{\gamma_2},L_{\gamma_3}]_L\Big), L_{[\eta]}, L_{[\delta]}]_L.
$$

For $\eta_1\in [\eta]$, $\delta_1\in [\delta]$, $x_{\gamma_i}\in L_{\gamma_i}$, $i=1,2,3$, $x_{\eta_1}\in L_{\eta_1}$, and $x_{\delta_1}\in L_{\delta_1}$, thanks to \eqref{eq:h1}, \eqref{eq:e2} and Theorem \ref{thm:1}-2),
\begin{equation*}
\begin{split}
&[[x_{\gamma_1},x_{\gamma_2},x_{\gamma_3}]_L,x_{\eta_1},x_{\delta_1}]_L
=[[x_{\gamma_1},x_{\gamma_2},x_{\gamma_3}]_L,\alpha(\alpha^{-1}(x_{\eta_1})),\alpha(\alpha^{-1}(x_{\delta_1}))]_L\\
=&[[x_{\gamma_1},\alpha^{-1}(x_{\eta_1}),\alpha^{-1}(x_{\delta_1})]_L,\alpha(x_{\gamma_2}),\alpha(x_{\gamma_3})]_L
+[\alpha(x_{\gamma_1}),[x_{\gamma_2},\alpha^{-1}(x_{\eta_1}),\alpha^{-1}(x_{\delta_1})]_L,\alpha(x_{\gamma_3})]_L\\
+&[\alpha(x_{\gamma_1}),\alpha(x_{\gamma_2}),[x_{\gamma_3},\alpha^{-1}(x_{\eta_1}),\alpha^{-1}(x_{\delta_1})]_L]_L\\
\subset&[[L_{\gamma_1},L_{\eta_1(\alpha,\alpha)},L_{\delta_1(\alpha,\alpha)}]_L,
L_{\gamma_2(\alpha^{-1},\alpha^{-1})},L_{\gamma_3(\alpha^{-1},\alpha^{-1})}]_L
+[L_{\gamma_1(\alpha^{-1},\alpha^{-1})},[L_{\gamma_2},L_{\eta_1(\alpha,\alpha)},L_{\delta_1(\alpha,\alpha)}]_L,L_{\gamma_3(\alpha^{-1},\alpha^{-1})}]_L\\
+&[L_{\gamma_1(\alpha^{-1},\alpha^{-1})},L_{\gamma_2(\alpha^{-1},\alpha^{-1})},[L_{\gamma_3},L_{\eta_1(\alpha,\alpha)},L_{\delta_1(\alpha,\alpha)}]_L]_L,
\end{split}
\end{equation*}
$$
[L_{\gamma_1},L_{\eta_1(\alpha,\alpha)},L_{\delta_1(\alpha,\alpha)}]_L=[L_{\gamma_2},L_{\eta_1(\alpha,\alpha)},L_{\delta_1(\alpha,\alpha)}]_L
=[L_{\gamma_3},L_{\eta_1(\alpha,\alpha)},L_{\delta_1(\alpha,\alpha)}]_L=0.
$$
Therefore,  $[[L_{\gamma_1},L_{\gamma_2},L_{\gamma_3}]_L,L_{\eta_1},L_{\delta_1}]_L=0$.

If there exists $\eta_1\in [\eta]$, $\delta_1\in [\delta]$ such that $[A_{-\gamma_1}L_{\gamma_1},L_{\eta_1},L_{\delta_1}]_L\neq 0$. By Definition \ref{defin:hl},
$$
[A_{-\gamma_1}L_{\gamma_1},L_{\eta_1},L_{\delta_1}]_L=[L_{\eta_1},L_{\delta_1},A_{-\gamma_1}L_{\gamma_1}]_L\subset \phi(A_{-\gamma_1})[L_{\gamma_1},L_{\eta_1},L_{\delta_1}]_L+\rho(L_{\eta_1},L_{\delta_1})(A_{-\gamma_1})\alpha(L_{\gamma_1}).
$$
Thanks to \eqref{eq:e2} and Theorem \ref{thm:1}-6), we have $[L_{\gamma_1},L_{\eta_1},L_{\delta_1}]_L=0$, and
$$\rho(L_{\eta_1},L_{\delta_1})(A_{-\gamma_1})\alpha(L_{\gamma_1})\subset A_{(\eta_1+\delta_1-\gamma_1)(\alpha^{-1},\alpha^{-1})}L_{\gamma_1(\alpha^{-1},\alpha^{-1})}\neq 0,$$
then $A_{(\eta_1+\delta_1-\gamma_1)(\alpha^{-1},\alpha^{-1})}\neq 0$ and $(\eta_1+\delta_1-\gamma_1)(\alpha^{-1},\alpha^{-1})\in \Lambda$.  We get that
 $\{\gamma_1,-\delta_1,\eta_1+\delta_1-\gamma_1\}$ is  a connection from $\gamma_1$ to $\eta_1$ and $\gamma_1 \sim \eta_1$, which is a contradiction.

Therefore, \begin{equation}\label{eq:e3}
\rho(L_{\eta_1},L_{\delta_1})(A_{-\gamma_1})\alpha(L_{\gamma_1})=0.
\end{equation}

It follows
$$
[L_{0,[\gamma]},L_{[\eta]},L_{[\delta]}]_L=0, ~~[L_{[\gamma]},L_{0,[\eta]},L_{[\delta]}]_L=0, ~~ [L_{[\gamma]},L_{[\eta]},L_{0,[\delta]}]_L=0.
$$

Finally, we consider the item $[L_{0,[\gamma]},L_{0,[\eta]},L_{[\delta]}]_L$ in \eqref{eq:e1}. For any $\delta_1\in [\delta]$, by \eqref{e:1} and \eqref{e:2},
\begin{equation*}
\begin{split}
&[L_{0,[\gamma]},L_{0,[\eta]},L_{[\delta]}]_L\\
=&[\sum_{\substack{{\gamma_1\in [\gamma]}\\-\gamma_1\in \Lambda}} A_{-\gamma_1}L_{\gamma_1} +
\Big(\sum_{\substack{
{\gamma_1,\gamma_2,\gamma_3\in [\gamma]}\\
 \gamma_1+\gamma_2+\gamma_3=0}} [L_{\gamma_1},L_{\gamma_2},L_{\gamma_3}]_L\Big),
\sum_{\substack{{\eta_1\in [\eta]}\\-\eta_1\in \Lambda}} A_{-\eta_1}L_{\eta_1} +\Big(\sum_{\substack{
{\eta_1,\eta_2,\eta_3\in [\gamma]}\\
 \eta_1+\eta_2+\eta_3=0}} [L_{\eta_1},L_{\eta_2},L_{\eta_3}]_L\Big),L_{[\delta]}]_L\\
 \subset&[A_{-\gamma_1}L_{\gamma_1},A_{-\eta_1}L_{\eta_1},L_{\delta_1}]_L+[A_{-\gamma_1}L_{\gamma_1},[L_{\eta_1},L_{\eta_2},L_{\eta_3}]_L,L_{\delta_1}]_L\\
 +&[[L_{\gamma_1},L_{\gamma_2},L_{\gamma_3}]_L,A_{-\eta_1}L_{\eta_1},L_{\delta_1}]_L+[[L_{\gamma_1},L_{\gamma_2},L_{\gamma_3}]_L,
 [L_{\eta_1},L_{\eta_2},L_{\eta_3}],L_{\delta_1}]_L.
\end{split}
\end{equation*}
Thanks to \eqref{eq:e2}, \eqref{eq:e3} and Definition \ref{defin:hl},

\vspace{2mm}$[[L_{\gamma_1},L_{\gamma_2},L_{\gamma_3}]_L,[L_{\eta_1},L_{\eta_2},L_{\eta_3}]_L,L_{\delta_1}]_L\subset [L_{[\gamma]},L_{[\eta]},L_{[\delta]}]_L=0,$

\vspace{2mm}
$[A_{-\gamma_1}L_{\gamma_1},[L_{\eta_1},L_{\eta_2},L_{\eta_3}]_L,L_{\delta_1}]_L
=[[L_{\gamma_1},L_{\gamma_2},L_{\gamma_3}]_L,A_{-\eta_1}L_{\eta_1},L_{\delta_1}]_L=0,
$
\begin{equation*}
\begin{split}
0&=[A_{-\gamma_1}L_{\gamma_1},A_{-\eta_1}L_{\eta_1},L_{\delta_1}]_L=[L_{\delta_1},A_{-\gamma_1}L_{\gamma_1},A_{-\eta_1}L_{\eta_1}]_L\\
&=\phi(A_{-\eta_1})[L_{\delta_1},A_{-\gamma_1}L_{\gamma_1},L_{\eta_1}]_L+\phi(A_{-\gamma_1})\rho(L_{\delta_1},L_{\gamma_1})(A_{-\eta_1})\alpha(L_{\eta_1})\\
&=\phi(A_{-\eta_1}A_{-\gamma_1})[L_{\eta_1},L_{\delta_1},L_{\gamma_1}]_L+\phi(A_{-\eta_1})\rho(L_{\eta_1},L_{\delta_1})(A_{-\gamma_1})\alpha(L_{\gamma_1})
+\phi(A_{-\gamma_1})\rho(L_{\delta_1},L_{\gamma_1})(A_{-\eta_1})\alpha(L_{\eta_1}).
\end{split}
\end{equation*}
Then $
[L_{0,[\gamma]},L_{0,[\eta]},L_{[\delta]}]_L=0,~~
[L_{0,[\gamma]},L_{[\eta]},L_{0,[\delta]}]_L=0, ~~[L_{[\gamma]},L_{0,[\eta]},L_{0,[\delta]}]_L=0. $

Therefore, $[I_{[\gamma]},I_{[\eta]},I_{[\delta]}]_L=0$.

By a complete similar discussion to that above, we have $[I_{[\gamma]},I_{[\gamma]},I_{[\delta]}]_L=0$.
\end{proof}

\begin{theorem}\label{thm:2} For any $[\gamma]\in \overline{\Gamma}$,
$I_{[\gamma]}$ is an ideal of split regular Hom 3-Lie-Rinehart algebra $(L, A)$.
\end{theorem}

\begin{proof}
By Lemma \ref{lemma:3} and \ref{lemma:po}, we have
\begin{equation*}
\begin{split}
&[I_{[\gamma]},L,L]_L=[I_{[\gamma]},H\oplus\bigoplus_{\xi\in [\gamma]}L_\xi\oplus\bigoplus_{\delta\notin [\gamma]}L_\delta,H\oplus\bigoplus_{\xi\in [\gamma]}L_\xi\oplus\bigoplus_{\delta\notin [\gamma]}L_\delta]_L\\
\subset&[I_{[\gamma]},H,H]_L+[I_{[\gamma]},H,L_{[\gamma]}]_L+[I_{[\gamma]},H,L_{\delta}]_L+[I_{[\gamma]},L_{[\gamma]},L_{[\gamma]}]_L\\
+&[I_{[\gamma]},L_{[\gamma]},L_{\delta}]_L+[I_{[\gamma]},L_{\delta},H]_L+[I_{[\gamma]},L_{\delta},L_{[\gamma]}]_L+[I_{[\gamma]},L_{\delta},L_{\delta}]_L
\subset I_{[\gamma]}.
\end{split}
\end{equation*}
Then $I_{[\gamma]}$ is a 3-Lie ideal of $L$, and
$$\rho(I_{[\gamma]},L)(A)L\subset \phi(A)[I_{[\gamma]},L,L]_L+[I_{[\gamma]},L,AL]_L\subset I_{[\gamma]}.$$
Follows from  Lemma \ref{lemma:3},  $I_{[\gamma]}$ is an ideal of split regular Hom 3-Lie-Rinehart algebra $(L,A)$.
\end{proof}

\

For a Hom 3-Lie-Rinehart algebra $(L,A,[~,~,~]_L,\phi,\alpha,\rho)$, denotes
\begin{equation}\label{eq:center}
Z_L(A):= \{ a ~|~a\in A, ax=0, \forall x\in L\},~~  Z_{\rho}(L):= \{ x ~|~x\in L, [x, L, L]_L=0, \rho(x, L)=0\}.
\end{equation}

Then  $Z_{\rho}(L)=Z(L)\cap Ker\rho$, and  $Z_{\rho}(L)$ is an ideal of $L$.

\begin{theorem}\label{thm:directsum}
If $Z_{\rho}(L)=0$ and $H=\big(\sum_{\substack{{\xi\in \Gamma}\\-\xi\in \Lambda}} A_{-\xi}L_\xi \big)+
\big(\sum_{\substack{
{\xi,\eta,\delta\in \Gamma}\\
 \xi+\eta+\delta=0}} [L_\xi,L_\eta,L_\delta]_L\big)$, then
 $$L=\bigoplus_{[\gamma]\in \overline {\Gamma}}I_{[\gamma]}.$$
\end{theorem}

\begin{proof}
Since $H=\big(\sum_{\substack{{\xi\in \Gamma}\\-\xi\in \Lambda}} A_{-\xi}L_\xi \big)+
\big(\sum_{\substack{
{\xi,\eta,\delta\in \Gamma}\\
 \xi+\eta+\delta=0}} [L_\xi,L_\eta,L_\delta]\big)$, by Definition \ref{defin:2} and Theorem \ref{thm:2}, $L=\sum_{[\gamma]\in \overline{\Gamma}}I_{[\gamma]}.$
Therefore, we only need to prove that  $\sum_{[\gamma]\in \bar{\Gamma}}I_{[\gamma]}=\bigoplus_{[\gamma]\in \overline{\Gamma}}I_{[\gamma]}$.

For $x\in I_{[\gamma]}\cap (\sum_{\substack{{\xi\in \overline{\Gamma}}\\ [\xi]\neq[\gamma]}} I_{[\xi]})$, by Lemma \ref{lemma:po},
 $$[x,L, L]_L\subset[x, I_{[\gamma]}, I_{[\gamma]}]_L+[x, I_{[\gamma]}, \sum_{\xi\notin[\gamma]} I_{[\xi]}]_L+[x, \sum_{\xi\notin[\gamma]} I_{[\xi]}, \sum_{\xi\notin[\gamma]} I_{[\xi]}]_L=0.$$
 It follows $x\in Z(L)$. Thanks to Definition \ref{defin:hl}, $\rho(x,L)=0$. Therefore, $x\in Z_{\rho}(L)=0$.
\end{proof}

\begin{coro}
Let $(L,A)$ be an indecomposable split regular Hom 3-Lie-Rinehart algebra with $Z_{\rho}(L)=0$, and $H=\big(\sum_{\substack{{\xi\in \Gamma}\\-\xi\in \Lambda}} A_{-\xi}L_\xi \big)+
\big(\sum_{\substack{
{\xi,\eta,\delta\in \Gamma}\\
 \xi+\eta+\delta=0}} [L_\xi,L_\eta,L_\delta]_L\big)$.  Then $\Gamma=[\gamma]$.
\end{coro}

\begin{proof}
Apply Theorem \ref{thm:directsum}.
\end{proof}

\begin{remark}
For a split regular Hom 3-Lie-Rinehart algebra $(L,A)$, since $A$ is a commutative associative algebra, then the decomposition of $A$ is similar to \cite{B1}. For $\lambda\in \Lambda$, we define
$$A_{0,[\lambda]}:=\big(\sum_{\beta\in [\lambda]} A_{-\beta}A_{\beta} \big)+
\big(\sum_{\substack{
{\xi,\eta\in \Gamma, \beta\in [\lambda]}\\
 \xi+\eta+\beta=0}} \rho(L_\xi,L_\eta)(A_\beta)\big),\quad   A_{[\lambda]}:=\bigoplus \limits_{\beta\in [\lambda]}A_\beta.$$
Thanks to Theorem \ref{thm:1},  $A_{0,[\lambda]}\subset A_0$. Denote
$\mathcal{A}_{[\lambda]}:=A_{0,[\lambda]}\oplus A_{[\lambda]}.$

By a similar discussion to the above, for $\lambda, \beta \in \Lambda$, if $[\lambda]\neq [\beta]$ then $\mathcal{A}_{[\lambda]}\mathcal{A}_{[\beta]}=0$. And if $Z_L(A)=0$ and $A_0=\big(\sum_{\lambda\in \Lambda} A_{-\lambda}A_{\lambda} \big)+
\big(\sum_{\substack{
{\xi,\eta\in \Gamma, \lambda\in \Lambda}\\
 \xi+\eta+\lambda=0}} \rho(L_\xi,L_\eta)(A_\lambda)\big),$ then  $A=\bigoplus_{[\lambda]\in \bar{\Lambda}}\mathcal{A}_{[\lambda]}$. Detailed properties of $A$ can be found in \cite{B1}.
\end{remark}

\begin{lemma}\label{lemma:ide1}
Let $(L,A)$ be a split regular Hom 3-Lie-Rinehart algebra, and $I$ be an ideal of $L$ satisfying  $\alpha(I)=I$. Then
\begin{equation}\label{eq:ideal}
I=(I\cap H)\oplus \bigoplus \limits_{\gamma\in \Gamma}(I\cap L_\gamma).
\end{equation}
\end{lemma}

\begin{proof}
Since $(L,A)$ is a split regular Hom 3-Lie-Rinehart algebra, we get $L=H\oplus \bigoplus_{\gamma\in \Gamma}L_{\gamma}$.

For any non-zero $x\in I$, if $x\in H$. Then $x\in I\cap H$.

If $x\notin H$, then there are $h_0\in H$ and non-zero $x_{\gamma_j}\in L_{\gamma_j}$, $1\leq j\leq m$, such that
 $x=h_0+\sum_{j=1}^m x_{\gamma_j}$.

 If $m=1$, then there are $h_1, h_2\in H$ such that $\gamma_1(h_1, h_2)\neq 0$. Then
 $$[h_1, h_2, x]_L=[h_1, h_2, h_0+x_{\gamma_1}]_L=\gamma_1(h_1, h_2)\alpha(x_{\gamma_1})\in I.$$
 Thanks to $\alpha(I)=I$ and $\gamma_1(h_1, h_2)\neq 0$, we get $x_{\gamma_1}\in I$, and $h_0\in H.$

If $m\geq 2$, by induction on $m$, suppose $\gamma_m(h_1, h_2)\neq 0$ for some $h_1, h_2\in H$. Then by \eqref{eq:s1},

$$\mbox{ad}^k(h_1, h_2)x=\gamma_1(h_1, h_2)^k\alpha(x_{\gamma_1})+\cdots+\gamma_m(h_1, h_2)^k\alpha(x_{\gamma_m}), \forall k\in Z^+.$$

Thanks to $\alpha (I)=I$ and Vandermonde determinant, we get $x_{\gamma_j}\in I$ for $1\leq j\leq m$ and $h_0\in H$.  Therefore,  \eqref{eq:ideal} holds.
\end{proof}

\begin{lemma}\label{lemma:ide2}
Let $(L,A)$ be a split regular Hom 3-Lie-Rinehart algebra. If an ideal  $I$ satisfies $I\subset H$, then $I\subset Z(L)$.
\end{lemma}

\begin{proof}
Thanks to Theorem \ref{thm:1} and $I\subset H$,  $I\cap L_{\gamma}=0$ for $\forall \gamma\in \Gamma,$ and
$$[I, H, L_{\gamma}]_L=0, \quad [I, L_{\gamma'}, L_{\gamma}]_L=0, \quad  \forall \gamma, \gamma'\in \Gamma.$$
Therefore,

$
[I,L,L]_L=[I,H\oplus (\bigoplus\limits_{\gamma\in \Gamma}L_{\gamma}),H\oplus (\bigoplus\limits_{\gamma\in \Gamma}L_{\gamma})]_L
$
$\subset [I, H,\bigoplus\limits_{\gamma\in \Gamma}L_{\gamma}]_L+ [I,\bigoplus\limits_{\gamma\in \Gamma}L_{\gamma},\bigoplus\limits_{\gamma\in \Gamma}L_{\gamma}]_L=0.
$
\end{proof}

\begin{lemma}\label{lemma:ide3}
If a split regular Hom 3-Lie-Rinehart algebra $(L,A)$ can be decomposed into the direct sum of  finite ideals $G_j$, $1\leq j\leq s$, that is, $L=G_1\oplus \cdots \oplus G_s$, then $\alpha(G_j)=G_j$, $1\leq j\leq s$.
\end{lemma}

\begin{proof}
Since $G_j$ are ideals of Hom 3-Lie-Rinehart algebra $(L,A)$, for any $1\leq j\leq s$, we have $\alpha(G_j)\subseteq G_j$.

 For any $x\in G_i$, there exists $y\in L$ such that $\alpha(y)=x$. Suppose $y=y_1+\cdots+y_s$, where  $y_j\in G_j$, $1\leq j\leq s$. Since $\alpha$ is an isomorphism,  $$\alpha(y)=\alpha(y_1)+\cdots+\alpha(y_s)=x\in G_i.$$
 Thanks to $\alpha (G_j)\subseteq G_j$ and $G_i\cap G_j=0$, we get   $\alpha(y_j)=0$, $y_j=0$, for $j\neq i$, that is, $y\in G_i$. Therefore, $\alpha(G_i)= G_i$, for $1\leq i\leq s$.
 \end{proof}

\begin{theorem}
If a  regular Hom 3-Lie-Rinehart algebra $(L,A, \phi, \alpha, \rho)$ can be decomposed into a direct sum of  finite ideals $G_j$, $1\leq j\leq s$,
and $Z_L(A)=0$,
then $(L,A,\phi,\alpha,\rho)$ is a split  regular Hom 3-Lie-Rinehart algebra with  a splitting Cartan subalgebra $H$ if and only if $(G_j, A, \phi, \alpha|_{G_j}, \rho|_{G_j\wedge G_j})$, $1\leq j\leq s$ are split  regular Hom 3-Lie-Rinehart algebras with  a splitting Cartan subalgebra $H_j$,  respectively, such that $H=\bigoplus_{j=1}^sH_j$, $\Gamma=\bigcup_{j=1}^s\Gamma_j$, where $\Gamma_j$  are root systems of $G_j$ associated to $H_j$, respectively.
\end{theorem}

\begin{proof}
 If $(L,A,\phi,\alpha,\rho)$ is a split regular Hom 3-Lie-Rinehart algebra, and $L=H\oplus \bigoplus_{\gamma\in\Gamma}L_{\gamma}$.  Thanks to Lemma \ref{lemma:ide1},
    $\alpha(G_j)=G_j$, and $$G_j=(H\cap G_j)\oplus \bigoplus_{\gamma\in\Gamma}(L_{\gamma}\cap G_j).$$

   We conclude that if $L_{\gamma}\cap G_j\neq 0 $ then $\gamma|_{H_j\wedge H_j}\neq 0.$

   In fact,  if $\gamma(H_j, H_j)=0.$ For any non-zero  $x\in L_{\gamma}\cap G_j$, from $H=\bigoplus_{j=1}^s(G_j\cap H)$, we have $[H, H, x]_L=[H_j, H_j, x]_L=\gamma(H_j, H_j)\alpha(x)=0$. Then
   $x\in H$, which is a contradiction. Therefore,

    $$G_j =H_j\oplus \bigoplus_{j_\gamma\in \Gamma_j}G_{j_\gamma}, 1\leq j\leq s,$$ where
    $H_j=G_j\cap H$ and $G_{j_\gamma}=G_j\cap L_{\gamma}\neq 0$,
    \begin{equation}\label{eq:Gammaj}
    \Gamma_j=\{~~~ j_\gamma=\gamma |_{H_j\wedge H_j}\quad |\quad \mbox{where} ~~\gamma\in \Gamma ~~~\mbox{satisfying}~ ~ G_j\cap L_{\gamma}\neq 0~~~\}.
    \end{equation}
By Definition \ref{defn:subandideal},  we have $\rho(G_j, L)AL\subset G_j$ and $\rho(G_j, G_j)A G_j\subset G_j$, $1\leq  j\leq s.$
Thanks to  $Z_L(A)=0$, $\rho(G_i, G_j)A=0$, $1\leq i\neq j\leq s.$ Then
$$A=A^j_0\oplus \bigoplus_{\lambda\in \Lambda_j}A_{\lambda}^j, \quad \mbox{where}\quad
 A^j_0=\{~~ a\in A ~~ | ~~ \rho(H_j, H_j)a=0~~\},$$
    $$A_{\lambda}^j=\{ a\in A ~~|~~ \rho(h_1, h_2)a=\lambda(h_1, h_2)\phi(a), ~~ \forall h_1, h_2\in H_j, \lambda\in \Lambda,~~ \lambda(H_j, H_j)\neq 0\},$$
\begin{equation}\label{eq:Lambdaj}
 \Lambda_j=\{~~\lambda|_{H_j\wedge H_j}, ~~\mbox{where}~~  \lambda\in \Lambda,~~ \lambda(H_j, H_j)\neq 0 \}.
 \end{equation}

  Therefore,   $(G_j, A, \phi, \alpha|_{G_j}, \rho|_{G_j\wedge G_j})$ is a split  regular Hom 3-Lie-Rinehart algebra  with a splitting Cartan subalgebra $H_j$ and the root system $\Gamma_j$ defined by \eqref{eq:Gammaj}, and the weight system $\Lambda_j$ defined by \eqref{eq:Lambdaj} associated to $H_j$.

 Conversely, if $(G_j, A, \phi, \alpha|_{G_j}, \rho|_{G_j\wedge G_j})$, $1\leq j\leq s$, are split  regular Hom 3-Lie-Rinehart algebras, and
  $G_j=H_j\oplus \bigoplus_{j_\gamma\in \Gamma_j}G_{j_\gamma}$, $1\leq j\leq s$  with a splitting Cartan subalgebra $H_j$ and the root system $\Gamma_j$ and the weight system $\Lambda_j$, respectively.

 Then $H=\bigoplus_{j=1}^s H_j$ is an abelian subalgebra of $L$.  For any $\gamma\in \Gamma_j$,  extending $\gamma: H_j\wedge H_j\rightarrow F$ to
 $\gamma: H\wedge H\rightarrow F$ ( still using the original symbol $\gamma$) by
\begin{equation*}
\gamma(h_1, h_2)=\left\{
\begin{split}
0,  & \quad h_1\not\in H_j~~ \mbox{or} ~~h_2\notin H_j,\\
\gamma(h_1, h_2),&\quad h_1, h_2\in H_j,\\
\end{split}
\right. \quad \forall h_1, h_2\in H,
\end{equation*}
we get that $L_{\gamma}=G_{j_\gamma}\neq 0$, and $\gamma\in (H\wedge H)^*_{\neq 0}$. Therefore, $H=\bigoplus_{j=1}^sH_j$ is a splitting Cartan subalgebra of $L$ with
the root system $\Gamma=\cup_{j=1}^s\Gamma_j,$ and
$$
L=H\oplus \bigoplus_{\gamma\in \Gamma} L_{\gamma}.$$

Thanks to Definition \ref{defn:subandideal} and $Z_L(A)=0$, $\rho(G_j, G_i)A=0$ for $1\leq i\neq j\leq s$, and  $\rho(G_j, G_j)A G_j\subset G_j$, $1\leq  j\leq s.$
Therefore, $\rho(H_j, H_j)AG_i=0$, $1\leq i\neq j\leq s$.

By a complete similar discussion to the above, for any $\lambda\in \Lambda_j$, extending $\lambda: H_j\wedge H_j\rightarrow F$ to
 $\lambda: H\wedge H\rightarrow F$ ( still using the original symbol $\lambda$) by
\begin{equation*}
\lambda(h_1, h_2)=\left\{
\begin{split}
0,  & \quad h_1\not\in H_j~~ \mbox{or} ~~h_2\notin H_j,\\
\lambda(h_1, h_2),&\quad h_1, h_2\in H_j,\\
\end{split}
\right. \quad \forall h_1, h_2\in H,
\end{equation*}
we get that $A_{\lambda}=A^j_{\lambda}\neq 0$, and $\lambda\in (H\wedge H)^*_{\neq 0}$. Therefore,  $\Lambda=\cup_{j=1}^s\Lambda_j\subset (H\wedge H)^*_{\neq 0},$ and
$$
A=A_0\oplus \bigoplus_{\lambda\in \Lambda} A_{\lambda}.
$$
where $A_0=\cap_{j=1}^sA^j_0.$ The proof is complete.

\end{proof}

\section*{Acknowledgements}

The first author (R. Bai)  was  supported by the Natural Science Foundation of
Hebei Province, China (A2018201126).

\bibliography{}

\end{document}